\documentclass[11pt,a4paper,reqno, draft]{amsart}
\usepackage{amsmath, amssymb, latexsym, enumerate,  graphicx,tikz, stmaryrd, eufrak,enumitem,MnSymbol}

\usepackage{thmtools}

\usepackage[pdftex]{hyperref}
\usepackage{cleveref}
\usepackage{bbm}
\usepackage{cases}
\usepackage{array}
\usepackage{tikz-cd}
\usepackage[all]{xy}

\usepackage[hmarginratio={1:1},vmarginratio={1:1},lmargin=80.0pt,tmargin=90.0pt]{geometry}

\usepackage{graphicx}
\usepackage[vcentermath]{youngtab}
\usepackage{nicefrac}

 \newlength{\baseunit}               
 \newcount{\numlines}                
\newtheorem{theorem}[subsubsection]{Theorem}
\newtheorem{lemma}[theorem]{Lemma}
\newtheorem{prop}[theorem]{Proposition}
\newtheorem{corollary}[subsubsection]{Corollary}

\theoremstyle{definition}
\newtheorem{definition}[subsubsection]{Definition}

\newtheorem{remark}[theorem]{Remark}

\newtheorem{example}[subsubsection]{Example}

\newcommand{\PSh}{\mathsf{PSh}}
\newcommand{\Sh}{\mathsf{Sh}}
\newcommand{\Ind}{\mathsf{Ind}}

\newcommand{\tto}{\twoheadrightarrow}

\newcommand{\Set}{\mathsf{Set}}
\newcommand{\Ab}{\mathsf{Ab}}
\newcommand{\Fun}{\mathsf{Fun}}
\newcommand{\Loc}{\mathsf{Cr}}
\newcommand{\Top}{\mathsf{Top}}
\newcommand{\QCoh}{\mathsf{QCoh}}

\newcommand{\ann}{\mathrm{ann}}

\newcommand{\Spec}{\mathrm{Spec}}

\newcommand{\Mod}{\mathsf{Mod}_k}
\newcommand{\Vecc}{\mathsf{Vec}}

\newcommand{\ba}{\mathbf{a}}
\newcommand{\bb}{\mathbf{b}}
\newcommand{\bc}{\mathbf{c}}
\newcommand{\bd}{\mathbf{d}}

\newcommand{\bj}{\mathbf{j}}
\newcommand{\bT}{\mathbf{T}}

\newcommand{\bD}{\mathbf{D}}
\newcommand{\bC}{\mathbf{C}}
\newcommand{\bB}{\mathbf{B}}
\newcommand{\bA}{\mathbf{A}}

\newcommand{\cS}{\mathcal{S}}

\newcommand{\cT}{\mathcal{T}}
\newcommand{\Co}{\mathcal{C}o}
\newcommand{\cV}{\mathcal{V}}
\newcommand{\cU}{\mathcal{U}}
\newcommand{\cM}{\mathcal{M}}
\newcommand{\cO}{\mathcal{O}}
\newcommand{\cL}{\mathcal{L}}

\newcommand{\preo}{\mathrm{pre}'}
\newcommand{\pre}{\mathrm{pre}}
\newcommand{\topp}{\mathrm{top}}

\newcommand{\id}{\mathrm{id}}
\newcommand{\Nat}{\mathrm{Nat}}
\newcommand{\el}{\mathsf{el}}
\newcommand{\coker}{\mathrm{coker}}
\newcommand{\Hom}{\mathrm{Hom}}
\newcommand{\ev}{\mathrm{ev}}
\newcommand{\co}{\mathrm{co}}
\newcommand{\Lan}{\mathrm{Lan}}
\newcommand{\Ob}{\mathrm{Ob}}
\newcommand{\Proj}{\mathrm{Proj}}
\newcommand{\im}{\mathrm{im}}
\newcommand{\op}{\mathrm{op}}
\newcommand{\colim}{\mathrm{colim}}

\newcommand{\mZ}{\mathbb{Z}}
\newcommand{\mA}{\mathbb{A}}
\newcommand{\mN}{\mathbb{N}}
\newcommand{\mX}{\mathbb{X}}
\newcommand{\mY}{\mathbb{Y}}
\newcommand{\mP}{\mathbb{P}}

\newcommand{\Yon}{\mathtt{Y}}
\newcommand{\tI}{\mathtt{I}}
\newcommand{\tS}{\mathtt{S}}
\newcommand{\tZ}{\mathtt{Z}}
\newcommand{\unit}{{\mathbf{1}}}

\begin{document}
\title[Additive pretopologies]{Additive Grothendieck pretopologies and presentations of tensor categories}
\author{Kevin Coulembier}

\address{School of Mathematics and Statistics, University of Sydney, NSW 2006, Australia}
\email{kevin.coulembier@sydney.edu.au}

\subjclass[2020]{18E10, 18E35, 	18F10, 18D15}

\keywords{Additive Grothendieck topology, Grothendieck category, noetherian and subcanonical topologies, tensor category} 

\begin{abstract}
We study how tensor categories can be presented in terms of rigid monoidal categories and Grothendieck topologies and show that such presentations lead to strong universal properties.
As the main tool in this study, we define a notion on preadditive categories which plays a role similar to (a generalisation of) the notion of a Grothendieck pretopology on an unenriched category. Each such additive pretopology defines an additive Grothendieck topology and suffices to define the sheaf category. This new notion also allows us to study the noetherian and subcanonical nature of additive topologies, to describe easily the join of a family of additive topologies and to identify  useful universal properties of the sheaf category. 
\end{abstract}

\maketitle


\section*{Introduction}

The notion of a Grothendieck topology on a small category $\bc$ allows one to define a category of sheaves and the development of sheaf cohomology, most notably \'etale cohomology on a scheme, see~\cite{SGA4, Ta}.
According to the original terminology, a Grothendieck topology is a collection of `sieves' satisfying suitable axioms. A Grothendieck {\em pre}topology is a collection of `coverings' $\{U_i\to U\,|\, i\}$, generalising the properties of open coverings of topological spaces. Each pretopology induces a topology, but many pretopologies induce the same topology.

 It has become customary to refer to pretopologies simply as topologies, see \cite{Ta}. From some points of view it can indeed be preferable to work with pretopologies. A presheaf $F:\bc^{\op}\to \Set$ is a sheaf for the induced topology if and only if, for every covering $\{U_i\to U\}$ in the pretopology, the diagram
\begin{equation}\label{classsheaf}F(U)\;\to\; \prod_i F(U_i)\;\rightrightarrows\; \prod_{j,l}F(U_j\times_UU_l)\end{equation}
is an equaliser. This recovers the familiar definition of a sheaf on a topological space.

The notion of Grothendieck topologies extends canonically to enriched categories, see \cite{BQ}. Furthermore, by \cite{BQ} and the Gabriel-Popescu theorem \cite{PG}, a category is Grothendieck abelian if and only if it can be realised as the category of additive sheaves on a preadditive site. 
On the other hand, the notion of a Grothendieck pretopology does not extend naively to the additive setting, see Appendix~\ref{Issues}. Furthermore, the existence of pullbacks to make sense of~\eqref{classsheaf} is not satisfied in the generality we want to apply our results. However, there are some examples of additive Grothendieck topologies where one can (and usually does) define sheaves with respect to exactness of certain sequences, similar to the classical sheaf condition~\eqref{classsheaf}. Firstly, let $\ba$ be a small abelian category. The ind-completion $\Ind\ba$ can be realised as the category of left exact functors $\ba^{\op}\to\Ab$.  Secondly, we can start from a category $\bc$ with a (non-enriched) pretopology and consider the preadditive category $\mZ\bc$ freely generated by $\bc$. By taking the difference of the two canonical maps we obtain, for every covering $\{U_i\to U\}$, a sequence
$$\amalg_{j,l}U_j \times_U U_l\;\to\;\amalg_i U_i \;\to\; U\;\to\;0,$$
where the coproducts are `formal' unless $\bc$ admits coproducts of sufficiently large cardinality. The Grothendieck category of abelian sheaves on $\bc$ can be realised as the category of additive functors $\mZ\bc^{\op}\to\Ab$ which send the above sequences to exact ones in $\Ab$.
 
We therefore define an additive pretopology on a preadditive category $\ba$ to be a class of formal sequences in $\ba$
$$\amalg_\gamma Z_\gamma\;\to\; \amalg_\beta Y_\beta\;\to \; X,$$
which satisfies two simple conditions. Our definition is much broader than what is needed to incorporate our two guiding examples above. The main advantage is that the union of two pretopologies will again be a pretopology. This feature, which is not satisfied for unenriched pretopologies, is useful for several applications. In Appendix~\ref{App} we show how the unenriched notion of pretopologies can be generalised to a form which one can almost naively linearise to obtain our definition of linear pretopologies.

We prove that each additive pretopology yields an additive topology and that every additive topology comes from some additive pretopology. Furthermore, sheaves with respect to the topology correspond precisely to additive functors $F:\ba^{\op}\to\Ab$ for which 
$$0\;\to\; F(X)\;\to\;\prod_\beta F(Y_\beta)\;\to \; \prod_\gamma F(Z_\gamma)$$
is exact for all sequences in the pretopology.

For the specific case of additive topologies which are both noetherian and subcanonical (see Definitions~\ref{DefNoe} and~\ref{defsub}) on a category $\ba$ which moreover is additive, our theory essentially recovers the theory of `ind-classes' as developed by Sch\"appi in \cite{Sch1, Sch2}. Our theory is directly inspired by the latter, and this is in particular the case for many of the techniques in Section~\ref{SecSheaves}.

Part of the motivation for the above constructions comes from the theory of tensor categories, in the sense of \cite{Del90, EGNO}, more specifically the theory of abelian envelopes, see \cite{EHS, AbEnv} and references therein. We investigate in which ways one can present a tensor category (or more precisely its ind-completion) via a generating rigid monoidal category and a monoidal (additive) Grothendieck topology. Using the above techniques we can reduce this to a manageable set of topologies and show that they come with strong universal properties, which will be crucial for applications in work in preparation. The main results here are Corollary~\ref{CorLow}, which significantly generalises \cite[Theorem~9.2.2]{EHS}, and Theorem~\ref{Thm3case}.

The paper is organised as follows. In Section~\ref{SecPrel} we recall the required background on Grothendieck topologies and categories. In Section~\ref{SecPT} we introduce additive pretopologies, establish the connection with additive topologies, prove the characterisation of sheaves and demonstrate that the notion of a pretopology provides a universal property for the sheaf category. In Section~\ref{PropTop} we study subcanonical and noetherian topologies, as inspired by the corresponding unenriched notions in 
\cite[I.\S 1.3 and I.\S 3.10]{Ta}. An additive topology is subcanonical if the representable presheaves are sheaves and noetherian if the sheafifications of the representable presheaves are compact. These notions are difficult to characterise directly from the topology, but are easily described on the level of pretopologies. We also introduce the notion of monoidal additive Grothendieck topologies.
We end the section with a discussion of the `tensor product' of two Grothendieck categories from~\cite{LRS}. In Section~\ref{SecTensor} we investigate the presentations of tensor categories. In Appendix~\ref{App}, we look back at non-enriched Gorthendieck pretopologies. We generalise this notion via some natural steps to arrive at a non-linear version of the notion we put forward as an additive Grothendieck pretopology.


\section{Preliminaries}\label{SecPrel}

\subsection{Notation and conventions}

\subsubsection{}\label{notation} We set $\mN=\{0,1,2,\ldots\}$. All categories we will consider are locally small. We will always {\em work over an unspecified commutative ring $k$.} For Sections~\ref{SecPrel} to \ref{PropTop}, essentially nothing is lost if one just assumes $k=\mZ$. We will also always denote by {\em $\ba$ an essentially small $k$-linear category.}

For a ring $R$ we will denote by $\underline{R}$ the (preadditive) category with one object $\bullet$ which has endomorphism ring $R$.

For a set $S$, we denote by $kS$ the free $k$-module on generators $\{e_s\,|\, s\in S\}$.
\subsubsection{}

For two $k$-linear categories $\bB,\bC$, we denote by $[\bB,\bC]$ the $k$-linear category of $k$-linear functors $\bB\to\bC$. In particular, we set
$$\PSh\ba\;:=\; [\ba^{\op},\Mod].$$
Note that $\PSh\ba$ (as an additive category) does not change if we replace $k$ by $\mZ$.


 We will typically leave out reference to $k$, so `functor', `sieve' and `(pre)topology' will refer to the $k$-linear versions (with the obvious exception of Appendix~\ref{App}).
We will use the symbol $\oplus$ for coproducts and reserve $\amalg$ for `formal coproducts'.

\subsubsection{} Recall that a cardinal $\kappa$ is {\bf regular} if the class of sets of cardinality strictly lower than $\kappa$ is closed under taking unions of strictly fewer than $\kappa$ sets. We only assume that a regular cardinal is bigger than $1$. In other words, we consider the regular cardinal $2$ and infinite regular cardinals.

\subsection{Grothendieck categories}

\subsubsection{}\label{DefGen}An object $G$ in a category $\bC$ is a \textbf{generator} if $\bC(G,-):\bC\to\Set$ is faithful. More generally, we call an essentially small full subcategory $\bc\subset\bC$ a generator if $\prod_{A\in\bc} \bC(A,-)$ is faithful.


A \textbf{Grothendieck category} $\bC$ is an AB5 abelian category which admits a generator. In other words, $\bC$ is abelian with a generator, $\bC$ admits set-indexed coproducts, and filtered colimits of short exact sequences in $\bC$ are exact. 
A Grothendieck category has enough injective objects and by Freyd's special adjoint functor theorem, a cocontinuous functor out of a Grothendieck category has a right adjoint.

The presheaf category
$\PSh\ba$ is a $k$-linear Grothendieck category, with generator $\ba$, identified with a full subcategory via the Yoneda embedding $\Yon:\ba\to\PSh\ba$.

\subsubsection{Localisations}\label{SecLoc}Let $\bB$ be a $k$-linear Grothendieck category. 
A fully faithful $k$-linear functor $\mathtt{I}:\bC\hookrightarrow\bB$, from a $k$-linear category $\bC$ is {\bf reflective} if it has a left adjoint $\tS:\bB\to\bC$, which we refer to as the sheafification. This terminology is a generalisation of the ordinary usage for localisations, see next paragraph, for which $\tS$ is an actual sheafification, see \ref{ThmBQ}-\ref{sheafif}. If $\tI$ is reflective, it follows that $\bC$ is cocomplete.

A reflective functor $\mathtt{I}$ is a {\bf localisation} if $\tS$ is (left) exact. If $\tI$ is a localisation, it follows that $\bC$ is a Grothendieck category itself. By a {\bf Giraud subcategory} of $\bB$ we mean a full replete subcategory for which the inclusion is a localisation. In other words, Giraud subcategories correspond to equivalence classes of localisations. Note that in the literature the notions of `localisation' and `Giraud subcategory' of $\bB$ appear interchangeably, but we use them for slightly distinct structures.

Besides the notation $\tI$ and $\tS$ for the inclusion and sheafification of a reflective subcategory $\bC$ of $\PSh\ba$, we will fix the notation $\tZ$ for the functor
$$\tZ=\tS\circ\Yon\,:\; \ba\,\to\,\bC.$$

By the Gabriel-Popescu theorem \cite{PG}, a Grothendieck category $\bC$ with generator $G\in\bC$ is a localisation of $\mathsf{Mod}_R$, for the ring $R:=\bC(G,G)$.
 It is often convenient to consider a slight generalisation, which recovers the original theorem for a one object category $\bc$.
\begin{theorem}[Gabriel - Popescu]\label{ThmGP}
For a Grothendieck category $\bC$ with generator $\bc\subset\bC$, the functor $\bC\to\PSh\bc: M\mapsto \bC(-,M)|_{\bc}$ is a localisation of $\PSh\bc$.
\end{theorem}

\subsubsection{Creators}\label{creator} 
Let $\bC$ be a Grothendieck category. By the right adjoint of a $k$-linear functor $u:\ba\to\bC$ we refer to the right adjoint of the left Kan extension $\Lan_{\Yon}u:\PSh\ba\to\bC$. This right adjoint is explicitly given by $\bC\to\PSh\ba$: $M\mapsto \bC(u-,M)$.

We say that such a $u:\ba\to\bC$  for which the right adjoint is a localisation (so in particular fully faithful) is a {\bf creator}.
In \cite{Lo} these functors are intrinsically characterised and an obvious necessary condition is that the essential image of $u$ is a generator in~$\bC$. Theorem~\ref{ThmGP} can then be rephrased as saying that the embedding of a generator is a creator.

\subsection{Compact objects} Let $\kappa$ be an infinite regular cardinal. Our main case of interest will be $\kappa=\aleph_0$, in which case we omit $\kappa$ from terminology. Let $\bC$ be a Grothendieck category. An object $X\in\bC$ is $\kappa$-compact (resp. $\kappa$-generated) if $\bC(X,-):\bC\to\Ab$ commutes with all $\kappa$-filtered colimits (resp. colimits of $\kappa$-filtered diagrams consisting of monomorphisms).

\begin{lemma}\label{LemCompact}
\begin{enumerate}[label=(\roman*)]
\item Let $A\in\bC$ be $\kappa$-generated. For every epimorphism $\oplus_{\beta\in B}M_\beta\tto A$ in $\bC$, there exists $B_0\subset B$ with $|B_0|<\kappa$ such that $\oplus_{\beta\in B_0}M_\beta\tto A$ is still an epimorphism.
\item Let $A\in\bC$ be $\kappa$-compact. For every exact sequence
$$\bigoplus_{\gamma\in C}N_\gamma\;\to\;\bigoplus_{\beta\in B}M_\beta\;\to\; A\;\to\; 0$$
in $\bC$ with each $N_\gamma$ $\kappa$-compact (or more generally such that $N_\gamma\to \bigoplus_{\beta\in B}M_\beta$ factors through a sub-coproduct of size less than $\kappa$, for each $\gamma$), there exist subsets $B_0\subset B$ and $C_0\subset C$ with $|B_0|<\kappa>|C_0|$ such that
restricting the summations to these subsets still yields an exact sequence.
\item The presheaf $F\in\PSh\ba$ is $\kappa$-compact if and only if there exists an exact sequence
$$\bigoplus_{\gamma\in C}\Yon(Z_\gamma)\;\to\;\bigoplus_{\beta\in B}\Yon(Y_\beta)\;\to\; F\;\to\; 0$$
with $|B|<\kappa>|C|$.
\item The presheaf $F\in\PSh\ba$ is $\kappa$-generated if and only if there exists an epimorphism
$\bigoplus_{\beta\in B}\Yon(Y_\beta)\tto F$
with $|B|<\kappa$.
\end{enumerate}
\end{lemma}
\begin{proof}
In part (i), for each $B'\subset B$ with $|B'|<\kappa$ denote by $A_{B'}$ the image of $\oplus_{\beta\in B'}M_\beta\to A$. Then $A$ is the ($\kappa$-filtered) colimit of these subobjects, so part (i) follows by definition.

Part (ii) follows similarly. For each $C'\subset C$ with $|C'|<\kappa$, we choose $B'\subset B$ with $|B'|<\kappa$ such that $\bigoplus_{\gamma\in C'}N_\gamma\to\bigoplus_{\beta\in B}M_\beta$ factors through $\bigoplus_{\beta\in B'}M_\beta$. This realises $A$ as a $\kappa$-filtered colimit of the cokernels of the above morphisms.

One direction of parts (iii) and (iv) follows from parts (i) and (ii). The other directions follow by using the fact that in $\Ab$, $\kappa$-filtered colimits commute with $\kappa$-small limits.
\end{proof}

\subsection{Linear Grothendieck topologies} We briefly review the theory of enriched topologies from \cite{BQ}, for the case of $k$-linear enrichment.

\subsubsection{}
For $A\in\ba$, a ($k$-linear) {\bf sieve on} $A$ is a $k$-linear subfunctor of $\Yon(A)=\ba(-,A)\in\PSh\ba$. For a sieve $R$ on $A$ and a morphism $f:B\to A$ in $\ba$, we denote by $f^{-1}R$ the sieve on $B$ which is the pullback of $R\to \ba(-,A)\leftarrow \ba(-,B)$ in $\PSh\ba$. Concretely, we have
$$g\in f^{-1}R(C)\;\Leftrightarrow\; f\circ g\in R(C),\qquad\mbox{for all $C\in A$ and $g\in\ba(C,B)$}.$$

A {\bf covering system} $\cT$ on $\ba$ is an assignment to each $A\in\ba$ of a set $\cT(A)$ of sieves on $A$.
\begin{definition}\label{DefT} A $k$-linear Grothendieck {\bf topology} is a covering system $\cT$ on $\ba$ such that for every $A\in\ba$:
\begin{enumerate}
\item[(T1)] We have $\ba(-,A)\in \cT(A)$;
\item[(T2)] For $R\in \cT(A)$ and a morphism $f:B\to A$  in $\ba$, we have $f^{-1}R\in \cT(B)$;
\item[(T3)] For a sieve $S$ on $A$ and $R\in\cT(A)$ such that for every $B\in\ba$ and $f\in R(B)\subset \ba(B,A)$ we have $f^{-1}S \in \cT(B)$, it follows that $S\in\cT(A)$.
\end{enumerate}
\end{definition}

\noindent A direct consequence of (T1) and (T3) is the following property:
\begin{itemize}
\item[(T4)] For a sieve $S$ on $A$ and $R\in\cT(A)$ such that $R\subset S$, it follows that $S\in\cT(A)$.
\end{itemize}
Similarly, a consequence of (T2) and (T3) is:
\begin{itemize}
\item[(T5)] If $R_1,R_2\in\cT(A)$, then $R_1\cap R_2 \in\cT(A)$.
\end{itemize}

Grothendieck topologies on $\ba$ are the same when regarding $\ba$ as a $k$-linear category or as a preadditive ($\mZ$-linear) category. It is just useful to take the $k$-linearity along for specific applications. It is common to refer to Grothendieck topologies on $\underline{R}$ as {\bf Gabriel topologies} on $R$, for a ring $R$.

\subsubsection{}\label{DefMeet}The class of covering systems, so also the set of Grothendieck topologies, is ordered by inclusion. We say that $\cT_1$ is a {\bf refinement} of $\cT$ if $\cT\subset\cT_1$, which means that $ \cT(A)\subset\cT_1(A)$ for all $A\in\ba$. It is clear that for a family of Grothendieck topologies $\{\cT_i\}$ the covering system $\cap_i\cT_i$ is again a Grothendieck topology. The same is not true for $\cup_i\cT_i$, see Example~\ref{Union}. However, by the previous property, we have a well-defined notion of the minimal (coarsest) Grothendieck topology containing all $\cT_i$, which is denoted by $\vee_i\cT_i$.

\begin{definition}\label{Def2}
For a $k$-linear Grothendieck topology $\cT$ on $\ba$, a presheaf $F\in\PSh\ba$ is a {\bf $\cT$-sheaf} if for every $A\in\ba$ and $R\in\cT(A)$, the morphism induced from $R\hookrightarrow \ba(-,A)$
$$F(A)\simeq \Nat(\ba(-,A),F)\to \Nat(R,F)$$
is an isomorphism. The full subcategory of $\PSh\ba$ of $\cT$-sheaves is denoted by $\Sh(\ba,\cT)$.
\end{definition}

When $F(A)\to \Nat(R,F)$ is a monomorphism for every $A\in\ba$ and $R\in\cT(A)$, we say that the presheaf $F$ is {\bf $\cT$-separated}.

The following theorem can be extracted from \cite[Theorem~1.5]{BQ} and its proof. 

\begin{theorem}[Borceux - Quinteiro]\label{ThmBQ}${}$
\begin{enumerate}[label=(\roman*)]
\item For each $k$-linear Grothendieck topology $\cT$ on $\ba$, the subcategory $\Sh(\ba,\cT)$ is a localisation of $\PSh\ba$.
\item For each localisation $\tI:\bC\hookrightarrow \PSh\ba$, the covering system $\cT$ on $\ba$ given by all sieves $R\subset \ba(-,X)$ which satisfy one of the equivalent conditions
\begin{enumerate}
\item $\Nat(\ba(-,X),\tI M)\to \Nat(R,\tI M)$ is an isomorphism for every $M\in\bC$;
\item $\tS$ maps $R\hookrightarrow \Yon(X)$ to an isomorphism; 
\end{enumerate}
is a $k$-linear Grothendieck topology.
\end{enumerate}
The above procedures give mutually inverse bijections between the set of $k$-linear Grothendieck topologies on $\ba$ and the set of Giraud subcategories of $\PSh\ba$.
\end{theorem}

\begin{remark}\label{RemEnv}
It follows from Theorem~\ref{ThmBQ}, or directly from Definition~\ref{DefT}, that there is a canonical bijection between the set of Grothendieck topologies on a $k$-linear category $\ba$ and its additive envelope.
\end{remark}

\begin{example}\label{Union}
Consider the two Gabriel topologies on $\mZ$ corresponding to the localisations $\mathsf{Mod}_{\mZ[1/p]}$ of $\Ab=\PSh\underline{\mZ}$ for two primes $p$ and observe that (T5) fails on their naive union.
\end{example}

\subsubsection{Sheafification}\label{sheafif} Fix a Grothendieck topology $\cT$ on $\ba$.
 For $X\in \ba$, the partial order $\le$ on $\cT(X)$ given by $R\le R'$ if $R'\subset R$ is directed by (T5). By \cite[Theorem~4.1]{BQ}, we have an endofunctor $\Sigma$ of $\PSh\ba$ such that
\begin{equation}\label{SigmaDL}\Sigma F(X)\;=\; \varinjlim_{R\in\cT(X)}\Nat(R,F),\qquad\mbox{for $F\in\PSh\ba$ and $X\in\ba$}\end{equation}
and by \cite[Theorem~4.4]{BQ} we have $\tI\circ \tS=\Sigma\circ\Sigma$.
There is an obvious natural transformation $\sigma:\id\to \Sigma$, and the unit $\eta:\id\to\tI\circ \tS$ of the adjunction $\tS \dashv \tI$ corresponds to $\sigma_{\Sigma}\circ\sigma:\id\to \Sigma\circ\Sigma$. Other elements of \cite[Theorems~4.1 and~4.4]{BQ} we need are:
\begin{enumerate}[label=(\roman*)]
\item $\Sigma:\PSh\ba\to\PSh\ba$ is left exact;
\item $F\in\PSh\ba$ is a sheaf if and only if $\sigma_F:F\to\Sigma F$ is an isomorphism;
\item $F\in\PSh\ba$ is separated if and only if $\sigma_F:F\to\Sigma F$ is a monomorphism;
\item $\Sigma F$ is separated for every presheaf $F$;
\item $\Sigma F$ is a sheaf, when $F$ is separated.
\end{enumerate}


\section{Linear pretopologies}\label{SecPT}

\subsection{Formal sequences}

\subsubsection{}
We will consider morphisms between formal coproducts of objects in $\ba$. These are actual morphisms in $\PSh\ba$ or in the formal completion of $\ba$ under small coproducts. In other words, {\bf (formal) morphisms} $f:\amalg_{\alpha \in X} A_\alpha\to \amalg_{\beta\in Y} B_\beta$, for families of objects $\{A_\alpha |\alpha\in X\}$ and $\{B_\beta|\beta\in Y\}$ in $\ba$ labelled by sets $X,Y$, are elements 
$$f=(f_{\alpha\beta})\;\in\; \prod_{\alpha\in X} \bigoplus_{\beta\in Y}\ba(A_\alpha,B_\beta).$$ 
To keep notation light we will almost always omit the labelling sets $X,Y$.

The composition $g\circ f$ for $g:\amalg_\beta B_\beta\to\amalg_\gamma C_\gamma$ is defined by  $(g\circ f)_{\alpha\gamma}=\sum_\beta g_{\beta\gamma}\circ f_{\alpha\beta}$, where by definition the sum is finite. For a fixed $\alpha$, we have $(g\circ f)_{\alpha\gamma}=0$ for all but finitely many $\gamma$, so $g\circ f$ is again a formal morphism.


For a (formal) morphism $v=(v_\alpha):\amalg_\alpha V_\alpha\to A$, we denote the sieve on $A$ generated by all $v_\alpha$ by $R_v\subset \ba(-,A)$. Clearly every sieve can be written in this way.


\subsubsection{} 
We are mainly interested in pairs $(p,q)$ of formal morphisms
\begin{equation}\label{EqSeq}\amalg_\gamma Z_\gamma\;\xrightarrow{p}\; \amalg_\beta Y_\beta\;\xrightarrow{q}\;X,\quad\mbox{with $q\circ p=0$.}\end{equation}
We will refer to them as {\bf (formal) sequences}.

For a regular cardinal $\kappa$, we call formal morphisms and sequences where the labelling sets in the coproducts are strictly bounded by $\kappa$  in cardinality, {\bf$\kappa$-bounded}. Similarly, a class of sequences is called $\kappa$-bounded if every sequence it contains is $\kappa$-bounded. 

The sequence \eqref{EqSeq} is {\bf right exact} if the induced sequence in $\Mod$
$$\prod_\gamma \ba(Z_\gamma,A)\;\leftarrow\; \prod_\beta \ba(Y_\beta,A)\;\leftarrow\;\ba(X,A)\;\leftarrow\;0$$
is exact for every $A\in\ba$. Similarly we say that $q$ is an {\bf epimorphism} if $ \ba(X,A)\to \prod_\beta \ba(Y_\beta,A)$ is injective for all $A$.
These notions of exactness do {\bf not} correspond to exactness in $\PSh\ba$. If the sequence is not formal ({\it i.e.} it is 2-bounded), we recover the ordinary notion of cokernels and epimorphisms in~$\ba$, which differ from the ones in $\PSh\ba$.

\subsubsection{}

For a class $\cS$ of sequences \eqref{EqSeq},
we denote by $\Co(\cS)$ the class of morphisms $q$ which appear on the right of the sequences. 
Next, set ${\Co}^1(\cS)=\Co(\cS)$ and denote by ${\Co}^0(\cS)$ the class of all identity morphisms in $\ba$.
For $i>1$ we define ${\Co}^i(\cS)$ iteratively as follows. If $v:\amalg_\alpha V_\alpha\to X$ is in ${\Co}^{i-1}(\cS)$ and for each $\alpha$ we have some $w(\alpha):\amalg_{\delta\in D(\alpha)} W_\delta(\alpha)\to V_\alpha$ in $\Co(\cS)$, then the collection of morphisms $v_\alpha\circ w(\alpha)_\delta$  form a formal morphism $\amalg_\alpha \amalg_{\delta \in D(\alpha)}W_{\delta}(\alpha)\to X$ which is in ${\Co}^i(\cS)$.
We also set 
$$\widetilde{\Co}(\cS)=\bigcup_{i\in\mN}{\Co}^i(\cS).$$

For $X\in\ba$, we denote by $\Co_X(\cS)\subset\Co(\cS)$ and $\widetilde{\Co}_X(\cS)\subset\widetilde{\Co}(\cS)$ the subclasses of morphisms with target $X$.

We also let $\Sh_{\cS}\ba$ be the full subcategory of $\PSh\ba$ of all $F$ for which
$$0\;\to\; F(X)\;\to\;\prod_\beta F(Y_\beta)\;\to \; \prod_\gamma F(Z_\gamma)$$
is exact in $\Mod$ for every $\amalg_\gamma Z_\gamma\to\amalg_\beta Y_\beta\to X$ in $\cS$.

\begin{lemma}\label{LemReflComp}
Let $\cS$ be a $\kappa$-bounded class of sequences \eqref{EqSeq}, for an infinite regular cardinal~$\kappa$. Consider the inclusion $\tI:\Sh_{\cS}\ba\hookrightarrow\PSh\ba$.
\begin{enumerate}[label=(\roman*)]
\item For every functor $J:\bj\to\Sh_{\cS}\ba$ from a $\kappa$-filtered category $\bj$, the colimit $\colim J$ exists and the canonical morphism $ \colim(\tI\circ J)\to\tI(\colim J)$ is an isomorphism.
\item The inclusion $\tI$ is reflective.
\item The image of $\tZ:\ba\to\Sh_{\cS}\ba$ consists of $\kappa$-compact objects.
\end{enumerate}
\end{lemma}
\begin{proof}
Since $\kappa$-small limits commute with $\kappa$-filtered colimits in $\Mod$, it follows that the presheaf $\colim (\tI\circ J)$ is actually contained in $\Sh_{\cS}\ba$. Part (i) is then a standard consequence. Part (ii) follows from the main result in \cite{AR} by part (i) and the facts that $\PSh\ba$ is locally presentable and $\Sh_{\cS}\ba$ is closed under limits in $\PSh\ba$. For part (iii), we observe that for $X\in\ba$ and a functor $J$ as in (i), we have
$$\Sh_{\cS}\ba(\tZ X,\colim J)\simeq \PSh\ba(\Yon X,\tI(\colim J))\simeq \colim (\tI\circ J)(X)\simeq \colim \Sh_{\cS}\ba(\tZ X,J),$$
which shows that $\tZ(X)$ is $\kappa$-compact.
\end{proof}


\subsection{Pretopologies versus topologies}
We refer to Appendix~\ref{App} for a detailed motivation for how the following is a linear analogue of a canonical generalisation of unenriched Grothendieck pretopologies.
\begin{definition}\label{DefPT} A {\bf $k$-linear Grothendieck pretopology} on $\ba$ is a  class $\cS$ of sequences of the form~\eqref{EqSeq} such that (PTa) and (PTb) are satisfied:
\begin{enumerate}
\item[(PTa)] 
For each  $q:\amalg_\beta Y_\beta\to X $ in $\Co(\cS)$ and morphism $f:A\to X$, there exists $q'\in \widetilde{\Co}_A(\cS)$ which admits a (formal) commutative diagram
$$\xymatrix{
\amalg_\beta Y_\beta\ar[rr]^{q}&& X\\
\amalg_\delta C_\delta \ar@{-->}[rr]^{q'}\ar@{-->}[u]^{f'}&& A\ar[u]^f.
}$$
\item[(PTb)] For every sequence \eqref{EqSeq} in $\cS$ and $f:A\to \amalg_\beta Y_\beta$ with $q\circ f=0$, there exists $q'\in \widetilde{\Co}_A(\cS)$ which admits a (formal) commutative diagram
$$\xymatrix{
\amalg_\gamma Z_\gamma\ar[r]^{p}& \amalg_\beta Y_\beta\ar[r]^{q}&X\\
\amalg_\alpha B_\alpha \ar@{-->}[r]^{q'}\ar@{-->}[u]^{f'}& A\ar[u]^f\ar[ru]_0.
}$$
\end{enumerate}
\end{definition}

\begin{remark}
The following are potential properties of a class $\cS$ of sequences \eqref{EqSeq}:
\begin{enumerate}
\item[(PTa')] For every $r:\amalg_\alpha V_\alpha\to X $ in $\widetilde{\Co}(\cS)$ and $f:A\to X$ in $\ba$, there exists $r'\in \widetilde{\Co}_A(\cS)$ which admits a (formal) commutative diagram
$$\xymatrix{
\amalg_\alpha V_\alpha\ar[rr]^r&& X\\
\amalg_\delta B_\delta \ar@{-->}[rr]^{r'}\ar@{-->}[u]^{f'}&& A\ar[u]^f.
}$$
\item[(PTb')]  For every sequence \eqref{EqSeq} in $\cS$ and $A\in\ba$, the following sequence in $\Mod$ is acyclic:
$$\bigoplus_\gamma \ba(A,Z_\gamma)\;\to\; \bigoplus_\beta \ba(A,Y_\beta)\;\to\;\ba(A,X).$$
\end{enumerate}
It is clear that (PTb') implies (PTb) and (PTa') implies (PTa). Furthermore, as proved below, (PTa) implies (PTa'), so (PTa') is satisfied for every pretopology.
\end{remark}

\begin{lemma}\label{LemPTa}Consider a class $\cS$ of sequences~\eqref{EqSeq} which satisfies {\rm (PTa)}.
\begin{enumerate}[label=(\roman*)]
\item Consider a set $\{\amalg_{\beta\in B(i)} Y_\beta^i\xrightarrow{q^i} X^i\,|\,i\in I\} $ of elements in $\Co(\cS)$.
For a given formal morphism $f=(f_i):A\to \amalg_{i\in I} X^i$, there exists $q'\in \widetilde{\Co}_A(\cS)$ which admits a commutative diagram as in {\rm (PTa)} with each pair $(q^i,f_i)$. 
\item Condition {\rm (PTa')} is satisfied for $\cS$.
\end{enumerate}\end{lemma}
\begin{proof}
By the definition of formal morphisms $f:A\to \amalg_i X^i$, it is sufficient to consider finite sets $I$ in (i). We then prove (i) by induction on $|I|$. The base case $|I|=1$ is precisely (PTa). Now assume that we have proved the claim for  $\{q^i:\amalg_\beta Y_\beta^i\to X^i\,|\,1\le i<n\} $ and consider $q^n:\amalg_\beta Y_\beta^n\to X^n$ in $\Co(\cS)$. By assumption, we have a commutative diagram
$$\xymatrix{
\amalg_{i<n}\amalg_\beta Y_\beta^i\ar[rr]^{\amalg_{i<n} q^i}&& \amalg_{i<n} X^i\\
\amalg_\delta C_\delta \ar@{-->}[rr]^{q''}\ar@{-->}[u]&& A\ar[u]\ar[rr]^{f_n}&& X^n&&\amalg_\beta Y^n_\beta\ar[ll]_{q^n}
}$$
for some $q''\in \widetilde{\Co}_A(\cS)$. We can now apply (PTa) to $q^n$ and each of the $f^n\circ q''_\delta:C_\delta\to X^n$ to get a suitable element of $\widetilde{\Co}_{C_\delta}(\cS)$ which we compose with $q''$ to find the desired element of  $\widetilde{\Co}_A(\cS)$.

Part (i) allows us to prove part (ii) iteratively. Indeed, by assumption, (PTa') is satisfied for every $r\in \Co^1(\cS)$. Assume that it is satisfied for a given $r:\amalg_\alpha V_\alpha\to X\in \Co^i(\cS)$ and construct an element in $\Co^{i+1}(\cS)$ by considering a morphism $\amalg_{\beta \in B(\alpha)} W_\beta(\alpha) \to V_\alpha$ in $\Co(\cS)$ for each $\alpha$. Starting from the diagram in (PTa') for $r$, we can now apply part (i) to each $B_\delta\to \amalg_\alpha V_\alpha$ and the collection $\{\amalg_{\beta} W_\beta(\alpha) \to V_\alpha|\alpha\}$.
\end{proof}

The class of pretopologies is canonically ordered with respect to inclusion. It follows immediately from the definition that for a family of pretopologies $\{\cS_i\}$, the class of sequences $\cup_i\cS_i$ is again a pretopology. However, for pretopologies $\cS_1$ and $\cS_2$, the class of sequences $\cS_1\cap\cS_2$ need not be a pretopology. This behaviour is dual to that of topologies, see \ref{DefMeet}. We will exploit this in Theorems~\ref{ThmTPT}(v) and~\ref{ThmLRS} below.


\subsubsection{}
For a ($k$-linear) pretopology $\cS$ we consider the covering system $\topp(\cS)$ of all sieves $R\subset\ba(-,X)$ which contain $R_r$ for some $r\in\widetilde{\Co}_X(\cS)$.

For a ($k$-linear) topology $\cT$ on $\ba$, we denote by $\preo(\cT)$ the class of all formal sequences \eqref{EqSeq} for which the associated sequence
\begin{equation}\label{eqYonseq}\bigoplus_\gamma\Yon(Z_\gamma)\;\to\;\bigoplus_\beta\Yon(Y_\beta)\;\to\; \Yon(X)\end{equation}
is acyclic in $\PSh\ba$ such that the image of the right morphism is an element in $\cT(X)$. 
The class $\preo(\cT)$ has the advantage of being defined directly from $(\ba,\cT)$. However, we will also need the class $\pre(\cT)$ of all formal sequences~\eqref{EqSeq} for which the induced sequence
$$\bigoplus_\gamma\tZ(Z_\gamma)\to\bigoplus_{\beta}\tZ(Y_\beta)\to \tZ(X)\to 0$$
is acyclic in $\Sh(\ba,\cT)$. By Theorem~\ref{ThmBQ}, we have $\preo(\cT)\subset\pre(\cT)$.

More generally, consider an exact cocontinuous functor $\Theta:\PSh\ba\to\bC$ to a cocomplete abelian category $\bC$. We denote by $\cS(\Theta)$ the class of formal sequences~\eqref{EqSeq}, for which $\Theta$ sends \eqref{eqYonseq} to a right exact sequence in $\bC$. We thus have $\pre(\cT)=\cS(\tS)$.

\begin{theorem}\label{ThmTPT}Let $\cT$ be a ($k$-linear) topology and $\cS$ a ($k$-linear) pretopology on $\ba$.
\begin{enumerate}[label=(\roman*)]
\item The classes of sequences $\pre(\cT)$ and $\pre'(\cT)$ are pretopologies.
\item The covering system $\topp(\cS)$ is a topology.
\item We have $\topp(\preo(\cT))=\cT=\topp(\pre(\cT))$.
\item The operations $\preo$ and $\topp$ are order (inclusion) preserving.
\item For a family of pretopologies $\{\cS_i\,|\, i\in I\}$, set $\cT_i=\topp \cS_i$. Then $\vee_{i\in I}\cT_i=\topp(\cup_{i\in I}\cS_i)$, so in particular 
$$\vee_{i\in I}\cT_i=\topp(\cup_{i\in I}\pre(\cT_i)).$$
\end{enumerate}
\end{theorem}

\begin{remark}\label{RemTPT}
\begin{enumerate}[label=(\roman*)]
\item As is immediate from the proof below, the observation in \ref{ThmTPT}(ii) that $\topp(\cS)$ is a Grothendieck topology only requires property (PTa) and not (PTb).
\item That $\cT\mapsto \pre(\cT)$ is also order preserving follows from Corollary~\ref{CorPOEquiv} below.
\end{enumerate}
\end{remark}

We start the proof of the theorem with the following lemma. 

\begin{lemma}\label{LemTPT}For a topology $\cT$ and a pretopology $\cS$, we have $\topp(\cS)\subset \cT$ if and only if $R_q\in\cT(X)$ for every $q\in \Co_X(\cS)$ and $X\in\ba$.
\end{lemma}
\begin{proof}
One direction is obvious. To prove the other direction, we assume that $R_q\in\cT(X)$ for every $q\in \Co(\cS)$. By (T4) it suffices to show that  $R_r\in\cT(X)$ for every $r\in \widetilde{\Co}_X(\cS)$. We prove this by induction on $i$ in $\widetilde{\Co}(\cS)=\cup_i\Co^i(\cS)$, where $i=0$ is fine by (T1) and $i=1$ is fine by assumption. Consider $v: \amalg_\alpha V_\alpha\to X$ in $\Co^i(\cS)$, and for each $\alpha$ we take $q(\alpha)$ in $\Co_{ V_\alpha}(\cS)$. We assume that $R_v\in \cT(X)$ and we need to show that $R_t\in \cT(X)$ for $t=v\circ\amalg_\alpha q(\alpha)$.

Take therefore arbitrary $B\in\ba$ and $f\in R_v(B)$. We have $f=\sum_\alpha v_\alpha \circ f_\alpha$ (finite sum) for certain $f_\alpha: B\to V_\alpha$. We have $\cap_\alpha f_\alpha^{-1}R_{q(\alpha)}\subset f^{-1}R_t$. Since $f_\alpha=0$ for all but finitely many $\alpha$, by (T2) and (T5) we find $\cap_\alpha f_\alpha^{-1}R_{q(\alpha)}\in \cT(B)$, so by (T4) $f^{-1}R_t\in\cT(B)$. That $R_t\in \cT(X)$ thus follows from (T3).
\end{proof}

\begin{lemma}\label{LemEpiLoc}
Consider an exact cocontinuous functor $\Theta:\PSh\ba\to \bC$ to a cocomplete abelian category $\bC$ and set $u:=\Theta\circ\Yon:\ba\to\bC$.

\begin{enumerate}[label=(\roman*)]
\item The class $\Co(\cS(\Theta))=\widetilde{\Co}(\cS(\Theta))$ comprises all morphisms $\amalg_\beta Y_\beta\to X$ for which the induced morphism $\oplus_\beta u(Y_\beta)\to u(X)$ is an epimorphism.
\item The class $\cS(\Theta)$ is a pretopology on $\ba$.
\end{enumerate}
\end{lemma}
\begin{proof}
Part (i) follows from the fact that $\ba\subset\PSh\ba$ is a generator, and the assumptions on $\Theta$.

Now we prove part (ii). Consider the solid diagram in (PTa). Let $P$ denote the pullback in $\PSh\ba$ of $q$ and $f$. Then $P$ is a quotient of some $\oplus_\delta\Yon(C_\delta)$. Composing the morphisms yields a commutative diagram as requested where it remains to be shown that the induced composite
$$\bigoplus_\delta u(C_\delta)\to \Theta(P)\to u(A)$$
is an epimorphism. That the left morphism is an epimorphism follows from right exactness of $\Theta$. Furthermore, left exactness of $\Theta$ implies that $\oplus_\beta u(Y_\beta)\leftarrow \Theta(P)\to  u(A)$ is a pullback. That $\Theta(P)\to  u(A)$ is an epimorphism therefore follows from the fact that $ u(q)$ was an epimorphism and \cite[Pullback Thm 2.54]{Freyd}. 
We can prove similarly that (PTb) is satisfied.
\end{proof}

\begin{proof}[Proof of Theorem~\ref{ThmTPT}] Part (i) for $\pre(\cT)$ is a special case of Lemma~\ref{LemEpiLoc}(ii).
That (PTa) is satisfied for $\preo(\cT)$ follows immediately from (T2). That (PTb') is satisfied follows from the fact that the exactness in (PTb') is just a reformulation of exactness in $\PSh\ba$. 

Now we prove that $\topp(\cS)$ is a topology. Condition (T1) follows by definition. Now take $R\in \topp(\cS)(X)$. By definition, we have $R_r\subset R$ for some $r\in\widetilde{\Co}_X(\cS)$. Using (PTa') allows us to conclude that for $f:A\to X$ we have $R_{r'}\subset f^{-1}R$ for some $r'\in\widetilde{\Co}_A(\cS)$. Hence (T2) follows. Finally, consider sieves $R,S$ on $X$ as in (T3). By assumption, we have $R_r\subset R$ for some $r:\amalg_\alpha V_\alpha\to X$ in $\widetilde{\Co}(\cS)$ and also $R_{t(\alpha)}\subset r_\alpha^{-1}S$ for some $t(\alpha)\in\widetilde{\Co}_{V_\alpha}(\cS)$, for every $\alpha$. But this means that $R_{r_\alpha\circ t(\alpha)}\subset S$ for every $\alpha$, and consequently $R_{s}\subset S$ for $s=r\circ (\amalg_\alpha t(\alpha))$. Hence $S$ is in $\topp(\cS)$. This concludes the proof of part (ii).

Part (iii) follows immediately from the definitions and property (T4) of $\cT$.
Part (iv) is immediate by construction.

Finally, we prove part (v). By (iv) we have $\cT_j\subset \topp(\cup_i\cS_i)$, for every $j\in I$. Now assume that $\cT_j\subset\cT$ for every $j$. In particular, $R_q\in \cT$ for every $q\in\cup_i \Co(\cS_i)=\Co(\cup_i\cS_i)$. By Lemma~\ref{LemTPT}, we thus find $\topp(\cup_i \cS_i)\subset\cT$, which concludes the proof.
\end{proof}

\begin{corollary}\label{Cor227}
\begin{enumerate}[label=(\roman*)]
\item For pretopologies $\cS_1$ and $\cS_2$ on $\ba$, we have $\topp(\cS_2)\subset\topp(\cS_1)$ if and only if
 for each $\amalg_\beta Y_\beta\to X$ in ${\Co}(\cS_2)$, there exist $r\in \widetilde{\Co}_X(\cS_1)$ with a commutative diagram
$$\xymatrix{
\amalg_\beta Y_\beta\ar[rr]&& X\\
\amalg_\alpha V_\alpha \ar@{-->}[u]\ar[rru]_r.
}$$
\item
Consider a topology $\cT$ and a pretopology $\cS$ on $\ba$ with $\cS\subset\pre\cT$. Then $\topp(\cS)=\cT$ if and only if for each formal morphism $\amalg_\beta Y_\beta\to X$ for which $\oplus_\beta \tZ(Y_\beta)\to Z(X)$ is an epimorphism, there exist $r\in \widetilde{\Co}_X(\cS)$ with a commutative diagram as in (i).\end{enumerate}
\end{corollary}
\begin{proof}
Part (i) is a special case of Lemma~\ref{LemTPT}. 
For part (ii), we have $\topp(\cS)\subset\cT$ by Theorem~\ref{ThmTPT}. It thus suffices to prove that $\cT\subset\topp\cS$ is equivalent to the condition in (ii).
By Lemma~\ref{LemEpiLoc}(i), this
 is the special case of part (i) for $\cS_1=\cS$ and $\cS_2=\pre(\cT)$.
\end{proof}

%

\begin{example}\label{ExamCommut}
For a commutative $k$-algebra $K$ we will denote formal sequences in $\underline{K}$ as sequences in $\mathsf{Mod}_K=\PSh\underline{K}$. 
For a set $\mathbf{x}=\{x_\alpha\in K\}$, consider the sequence
$$s_{\bf x}\;:\quad\bigoplus_{\alpha,\beta|\alpha\not=\beta} K\;\to\;\bigoplus_{\alpha}K\;\to\; K$$
where the right morphisms come from $1\mapsto x_\alpha$ and the left morphisms send $1$ in the $(\alpha,\beta)$-labelled copy of $K$ to $0$ everywhere, except to $x_\beta$ in the $\alpha$-copy of $K$ and to $-x_\alpha$ in the $\beta$-copy. For any collection $E$ of such sets ${\bf x}$, the collection $\{s_{\bf x}\,|\,{\bf x}\in E\}$ is a pretopology. This gives a unified construction of a pretopology for every Gabriel topology on $K$.
\begin{enumerate}[label=(\roman*)]
\item If we take a set $E$ of elements in $K$, then $\topp(\{s_{\{x\}}|x\in E\})$ is the topology corresponding to the localisation $\mathsf{Mod}_{K_E}$ of ${\mathsf{Mod}_K}$.
\item Let $k$ be a field, and $K=k[x,y]$. Then $\topp(\{s_{\{x,y\}}\})$ is the topology corresponding to the localisation $\QCoh\mX$ of $\mathsf{Mod}_{k[x,y]}$ with $\mX=\mA^2\backslash\{0\}$.
\end{enumerate}
\end{example}

\subsection{Sheaves}\label{SecSheaves}

\begin{theorem}\label{ThmSheaves}
For a pretopology $\cS$, set $\cT:=\topp(\cS)$.  Then $\Sh_{\cS}\ba=\Sh(\ba,\topp(\cS))$. So in particular, $\Sh_{\cS}\ba$ is a Grothendieck category. 
\end{theorem}

We will write the proof of the theorem in a couple of steps. Throughout, we keep the pretopology $\cS$ fixed and always set $\cT=\topp(\cS)$.

For a morphism $v:\amalg_\alpha V_\alpha\to X$ we can complete $\oplus_\alpha\Yon(V_\alpha)\tto R_v$ to a presentation of $R_v$ in $\PSh\ba$, so that applying $\Nat(-,F)$, for $F\in\PSh\ba$,  yields the exact sequence
\begin{equation}\label{NatPres}0\to \Nat(R_v,F)\to \prod_\alpha F(V_\alpha)\to \prod_{f: W\to \amalg_\alpha V_\alpha\,|\, v\circ f=0}F(W).\end{equation}

\begin{lemma}\label{LemSep}
The following conditions are equivalent for $F\in\PSh\ba$.
\begin{enumerate}[label=(\alph*)]
\item $F$ is $\cT$-separated.
\item $F(X)\to\prod_\alpha F(V_\alpha)$ is a monomorphism for every $\amalg_\alpha V_\alpha\to X$ in $\widetilde{\Co}(\cS)$.
\item $F(X)\to\prod_\beta F(Y_\beta)$ is a monomorphism for every $\amalg_\beta Y_\beta\to X$ in $\Co(\cS)$.
\end{enumerate}
\end{lemma} 
\begin{proof}
Since compositions of monomorphisms are monomorphisms, it is clear that (b) and (c) are equivalent. We thus focus on the equivalence between (a) and (b).

It follows from the definition of $\topp(\cS)$ that $F$ is separated if and only if $F(X)\to \Nat(R_s,F)$ is a monomorphism for every $s\in \widetilde{\Co}(\cS)$. The presentation \eqref{NatPres} shows that the latter is indeed equivalent with condition (b).
\end{proof}

\begin{lemma}\label{LemChiant}
Let $F$ be a $\cT$-separated presheaf.
\begin{enumerate}[label=(\roman*)]
\item For every sequence \eqref{EqSeq} in $\cS$, we have
$$\ker\left(\prod_\beta F(Y_\beta)\;\to \; \prod_{f: U\to \amalg_\beta Y_\beta\,|\, q\circ f=0} F(U)\right)\;=\;\ker\left(\prod_\beta F(Y_\beta)\;\to \; \prod_\gamma F(Z_\gamma)\right).$$
\item Consider $v:\amalg_\alpha V_\alpha\to X$ in $\widetilde{\Co}(\cS)$ and a morphism $q(\alpha):\amalg_{\beta\in B(\alpha)} W_\beta(\alpha)\to V_\alpha$ in $\Co(\cS)$, for each $\alpha$, and set $w=v\circ \amalg_\alpha q(\alpha)$.
The kernel of
$$ \prod_\alpha F(V_\alpha)\;\to \;\prod_{f:U\to\amalg_\alpha V_\alpha \,|\, v\circ f=0} F(U)$$
equals the kernel of the composite of
$$\prod_\alpha F(V_\alpha)\;\to\;\prod_{\alpha}\prod_\beta F(W_\beta(\alpha))\;\to\; \prod_{g :Q\to\amalg_\alpha\amalg_\beta W_\beta(\alpha) \,|\, w\circ g=0} F(Q).$$
\item For $r:\amalg_\alpha V_\alpha\to X$ in $\widetilde{\Co}(\cS)$ and any sieve $R$ on $X$ containing $R_r$, the inclusion $R_r\subset R$ induces a monomorphism $\Nat(R,F)\hookrightarrow \Nat(R_r,F)$.
\end{enumerate}
\end{lemma}
\begin{proof}
For both proposed equalities in (i) and (ii), inclusion in one direction is obvious, so we only prove inclusion in the other direction.

For part (i), take $y\in\prod_\beta F(Y_\beta)$ which is sent to zero in $\prod_\gamma F(Z_\gamma)$. We need to show it is also in the kernel on the left-hand side. So consider $f: U\to \amalg_\beta Y_\beta$ with $  q\circ f=0$. By assumption (PTb) there exists a commutative diagram with $q':\amalg_\alpha B_\alpha\to U$ in $\widetilde{\Co}(\cS)$ such that evaluation of $F$ yields a commutative diagram
$$\xymatrix{
\prod_\gamma F(Z_\gamma)\ar[d]& \prod_\beta F(Y_\beta)\ar[l]\ar[d]&&&0\ar@{{|}->}[d]&y\ar@{{|}->}[l]\ar@{{|}->}[d]\\
\prod_\alpha F(B_\alpha) & F(U)\ar@{_{(}->}[l]&&&0&?\ar@{{|}->}[l].
}$$
The lower horizontal arrow is a monomorphism by assumption that $F$ be separated, see Lemma~\ref{LemSep}.  Hence $y$ is indeed sent to zero in $F(U)$ as desired.

Now we prove part (ii). Consider $x\in \prod_\alpha F(V_\alpha)$ sent to zero by the composite morphism. We need to show it is also in the first kernel. Consider first an arbitrary $f:U\to\amalg_\alpha V_\alpha$. By Lemma~\ref{LemPTa}(i) there exists some $q': \amalg_\delta C_\delta\to U$ in $\widetilde{\Co}(\cS)$, such that evaluation of $F$ yields a commutative diagram
$$\xymatrix{
\prod_\alpha\prod_\beta F(W_\beta(\alpha))\ar[d]&&\prod_\alpha F(V_\alpha)\ar[ll]\ar[d]\\
\prod_\delta F(C_\delta)&&F(U).\ar@{_{(}->}[ll]
}$$ Now assume that $ v\circ f=0$.
It then follows from that $x$ is sent to zero by the upper path from top right to bottom left in the diagram. Consequently $x\mapsto 0$ under $\prod_\alpha F(V_\alpha)\to F(U)$. This concludes the proof of part (ii).

Now we prove part (iii). 
Consider arbitrary $A\in\ba$ and $f\in R(A)\subset\ba(A,X)$. By (PTa') we have an $r':\amalg_\beta B_\beta\to A$ in $\widetilde{\Co}(\cS)$ yielding a commutative diagram in $\PSh\ba$
$$\xymatrix{
\bigoplus_\alpha\Yon(V_\alpha)\ar[r]&R_r\ar@{^{(}->}[r]&R\ar@{^{(}->}[r]&\Yon(X)\\
\bigoplus_\beta\Yon(B_\beta)\ar[rr]\ar[u]&&\Yon(A).\ar[u]
}$$
Applying $\Nat(-,F)$ and the Yoneda lemma yields the following commutative diagram:
$$\xymatrix{
\prod_\alpha F(V_\alpha)\ar[d]&\Nat(R_r,F)\ar[l]&\Nat(R,F)\ar[l]\ar[d]^{\eta\mapsto \eta_A(f)}\\\
\prod_\beta F(B_\beta)&&F(A).\ar@{_{(}->}[ll]
}$$
If $\eta\in\Nat(R,F)$ is sent to $0$ in $\Nat(R_r,F)$, it thus follows that $\eta_A(f)=0$. Since $f\in R(A)$ was arbitrary, $\eta_A:R(A)\to F(A)$ is zero. But also $A$ was arbitrary, so $\eta=0$.
\end{proof}

\begin{corollary}\label{CorChiant}
A $\cT$-separated presheaf $F$ is a $\cT$-sheaf if and only if the sequence
$$ F(X)\to \prod_\alpha F(V_\alpha)\to \prod_{f: W\to \amalg_\alpha V_\alpha\,|\, r\circ f=0}F(W)$$
is exact for every $r:\amalg_\alpha V_\alpha\to X$ in $\widetilde{\Co}(\cS)$.
\end{corollary}
\begin{proof}
By Lemma~\ref{LemSep} and exactness in \eqref{NatPres}, exactness of the displayed sequence is equivalent with $F(X)\to\Nat(R_r,F)$ being an isomorphism, for every $r\in \widetilde{\Co}_X(\cS)$. In particular, the sequence is exact for a sheaf. Now for any $R\in\cT(X)$ we have $R_r\subset R\subset \ba(-,X)$ for some $r\in \widetilde{\Co}(\cS)$. By Lemma~\ref{LemChiant}(iii), the inclusions yield monomorphisms
$$F(X)\hookrightarrow \Nat(R,F)\hookrightarrow \Nat(R_r,F).$$
Consequently, if the composite is an isomorphism, then so is the left arrow. This proves the second direction of the claim.
\end{proof}

\begin{proof}[Proof of Theorem~\ref{ThmSheaves}]
Based on Lemma~\ref{LemSep} and Corollary~\ref{CorChiant} it now suffices to prove that the following properties are equivalent for a separated presheaf $F$:
\begin{enumerate}
\item[(p)] The sequence 
$$F(X)\;\to\;\prod_\beta F(Y_\beta)\;\to \; \prod_\gamma F(Z_\gamma)$$
is exact for every $\amalg_\gamma Z_\gamma\to\amalg_\beta Y_\beta\to X$ in $\cS$.
\item[(t)] The sequence
$$ F(X)\to \prod_\alpha F(V_\alpha)\to \prod_{f: U\to \amalg_\alpha V_\alpha\,|\, r\circ f=0}F(U)$$
is exact for every $r:\amalg_\alpha V_\alpha\to X$ in $\widetilde{\Co}(\cS)$.
\end{enumerate}

If (t) is satisfied, then considering the special case $r\in\Co(\cS)$ together with Lemma~\ref{LemChiant}(i) shows that (p) is also satisfied.

Now assume that (p) is satisfied for a separated presheaf $F$. We will prove by induction on $i\in\mN$ that (t) is satisfied for $r\in \Co^i(\cS)$. For $i=0$ the statement is trivial and for $i=1$ it is immediate (or follows from \ref{LemChiant}(i)).

Assume therefore that for some $r:\amalg_\alpha V_\alpha\to X$ in $\Co^{i}(\cS)$, the sequence in (t) is exact. Now take for each $\alpha$ a sequence
$$\amalg_\gamma Z_\gamma(\alpha)\xrightarrow{p(\alpha)} \amalg_\beta W_\beta(\alpha)\xrightarrow{q(\alpha)} V_\alpha$$
in $\cS$ and consider $w=r\circ\amalg_\alpha q(\alpha)\in \Co^{i+1}(\cS)$. We need to prove that 
$$F(X)\to \prod_\alpha \prod_\beta F(W_\beta(\alpha))\to \prod_{g :Q\to\amalg_\alpha\amalg_\beta W_\beta(\alpha) \,|\, w\circ g=0} F(Q)$$
is exact. We will exploit that the left morphism factors via $F(X)\to\prod_\alpha F(V_\alpha)$ to prove exactness in two steps.


Take $m\in \prod_\alpha\prod_\beta F(W_\beta(\alpha))$ such that $m\mapsto 0$ in the above sequence. 
By letting $g$ range over all morphisms of the composite form $ Z_{\gamma_0}(\alpha_0)\to\amalg_\beta W_\beta(\alpha_0)\to \amalg_\alpha \amalg_\beta W_\beta(\alpha)$ and using assumption (p) for each $\alpha_0$ (and the fact that products are exact in $\Mod$), we find that $m$ is the image of some $n$ under $\prod_\alpha F(V_\alpha)\to\prod_\alpha\prod_\beta F(W_\beta(\alpha))$. That $n$ is in the image of $F(X)\to\prod_\alpha F(V_\alpha)$ then follows from Lemma~\ref{LemChiant}(ii) and our assumption on $r$.
\end{proof}

We conclude the subsection with some consequences of our results.


\begin{corollary}\label{CorTriv} 
Consider a class $\cS$ of sequences \eqref{EqSeq} for which $\Sh_{\cS}\ba\hookrightarrow \PSh\ba$ is reflective (for instance $\cS$ is a pretopology or $\kappa$-bounded).
Then all the induced sequences
$$\bigoplus_\gamma \tZ(Z_\gamma)\to\bigoplus_{\beta}\tZ(Y_\beta)\to \tZ(X)\to 0$$
are exact in $\Sh_{\cS}\ba$. In particular,  for a pretopology $\cS$ we have $\cS\subset\pre(\topp\cS)$.
\end{corollary}
\begin{proof}
Exactness of the sequences follows from the defining condition of $\Sh_{\cS}\ba$ and the adjunction $\tS\dashv \tI$ which implies
$$\Sh_{\cS}\ba (\tZ(-),F)\simeq F(-).$$ The second statement then follows from Theorem~\ref{ThmSheaves} .
\end{proof}

\begin{corollary}
Consider a creator $u:\ba\to\bC$ of a Grothendieck category $\bC$. Then $\bC$ is equivalent to the full subcategory of $F\in\PSh\ba$ for which
$$0\;\to\; F(X)\;\to\;\prod_\beta F(Y_\beta)\;\to \; \prod_\gamma F(Z_\gamma)$$
is exact in $\Mod$ for every sequence \eqref{EqSeq} for which the following is exact in $\bC$:
$$\bigoplus_\gamma u(Z_\gamma)\;\to\;\bigoplus_\beta u(Y_\beta)\;\to\; u(X)\;\to\;0.$$
\end{corollary}
\begin{proof}
This is a combination of Theorem~\ref{ThmBQ} and Theorem~\ref{ThmSheaves} applied to $\cS=\pre(\cT)$, and Theorem~\ref{ThmTPT}(iii).
\end{proof}

\begin{corollary}\cite[Proposition~2.8]{LRS}\label{IntGiraud}
The intersection of a family of Giraud subcategories of $\PSh\ba$ is again Giraud. Concretely, for a family of topologies $\{\cT_i\,|\,i\in I\}$ on $\ba$, we have
$$\bigcap_{i\in I}\Sh(\ba,\cT_i)\;=\;\Sh(\ba,\vee_{i\in I}\cT_i).$$
\end{corollary}
\begin{proof}
By definition, for pretopologies $\cS_i$, the intersection of the subcategories $\Sh_{\cS_i}\ba$ equals $\Sh_{\cup\cS_i}\ba$.
The conclusion thus follows from Theorems~\ref{ThmSheaves} and~\ref{ThmTPT}(iii) and (v).
\end{proof}

\begin{remark}
The condition on a class $\cS$ of sequences~\eqref{EqSeq} to be a pretopology is not necessary for $\Sh_{\cS}\ba$ to be a Giraud subcategory of $\PSh\ba$. For instance, we can have $\Sh_{\cS}\ba=0$ while $\cS$ is not a pretopology. For example:

Set $\ba=\underline{R}$, for $R$ the free $k$-algebra on generators $a,b$ with relation $ab=0$. Then the pair of sequences
$$\bullet\xrightarrow{0} \bullet\xrightarrow{a} \bullet\quad\mbox{and}\quad \bullet\xrightarrow{0} \bullet\xrightarrow{b}\bullet$$
does not satisfy (PTa) or (PTb), but the only right $R$-module on which both $a$ and $b$ acts as isomorphisms is zero.
\end{remark}

\subsection{Universal property}Unless further specified, $\bB$ and $\bC$ are $k$-linear cocomplete categories. Denote the category of $k$-linear cocontinuous functors $\bB\to\bC$ by $[\bB,\bC]_{cc}$.

\subsubsection{}\label{Skappa}For a class $\cS$ of sequences \eqref{EqSeq} in $\ba$, we denote by $[\ba,\bB]_{\cS}$ the category of $k$-linear functors $h:\ba\to\bB$ for which
$$\bigoplus_\gamma h(Z_\gamma)\;\to\;\bigoplus_\beta h(Y_\beta)\;\to\; h(X)\;\to\; 0$$
is exact (is a cokernel diagram) in $\bB$, for every sequence~\eqref{EqSeq} in $\cS$. 


The following type of result is standard, see for instance \cite[\S 6.4]{Ke} or \cite[\S 3]{Sch2}.
\begin{prop}\label{PropUni}
Let  $\cS$ be a class of sequences \eqref{EqSeq} for which $\Sh_{\cS}\ba\hookrightarrow\PSh\ba$ is reflective (for instance $\cS$ is $\kappa$-bounded or a pretopology). Then precomposition with $\tZ:\ba\to\Sh_{\cS}\ba$ yields an equivalence
$$-\circ\tZ:\quad[\Sh_{\cS}\ba,\bB]_{cc}\;\xrightarrow{\sim}\; [\ba,\bB]_{\cS}.$$
\end{prop}

We start the proof with the following well-known lemma.
\begin{lemma}\label{LemYon}
Assume that $\ba$ is additive. For $F\in\PSh\ba$ let $\el F$ be the category elements of $F$, that is the category of pairs $(A,p)$ with $A\in\ba$ and $p\in F(A)$ with morphisms $(A,p)\to (B,q)$ given by morphisms $f:A\to B$ for which $F(f)$ sends $q$ to $p$. For the forgetful functor $J:\el F\to \ba$, we have a canonical isomorphism
$$\colim(\Yon\circ J)\;\xrightarrow{\sim}\; F.$$ 
\end{lemma}
\begin{proof}
We verify that $F$ possesses the necessary universal property.
Consider therefore $G\in \PSh\ba$ together with, for every $A\in\ba$ and $a\in F(A)$, a given associated $a'\in G(A)$ such that, for every $f:A\to B$ in $\ba$, we have $G(f)(b')=a'$ whenever
 $F(f)(b)=a$. Due to the Yoneda lemma, we just need to prove that there exists a unique natural transformation $F\to G$ which sends every $a\in F(A)$ to $a'\in G(A)$. Uniqueness is obvious, we only need to demonstrate that the functions $F(A)\to G(A)$, $a\mapsto a'$ are $k$-linear and the collection forms a natural transformation.
 
 For $(a_1,a_2)$ in $F(A)^{\times 2}\simeq F(A\oplus A)$, we find that $a_1+a_2$ is the image of $(a_1,a_2)$ under $F$ acting on the diagonal $A\to (A\oplus A)$. By assumption the same must be true for $a_1',a_2'$ and $G$, whence $(a_1+a_2)'=a_1'+a_2'$. That $(\lambda a)'=\lambda a'$ for $\lambda\in k$ follows similarly but more easily. Naturality is also satisfied by definition.
\end{proof}

\begin{remark}
Lemma~\ref{LemYon} does not remain true without the additivity assumption on $\ba$. For example, let $k$ be a field and set $\ba=\underline{k}$. With notation from~\ref{notation}, for a vector space $V\in \PSh\underline{k}$, the colimit in Lemma~\ref{LemYon} is the cokernel of
$$k({V\times k})\;\to\; kV,\quad e_{(v,\lambda)}\mapsto \lambda e_v-e_{\lambda v}.$$
For $|\cdot|:kV\to k$ given by $\sum_v\lambda_ve_v\mapsto \sum_v\lambda_v$, this cokernel is given by
$$k{V}\;\tto\; k{\mP V}, \quad e_v\mapsto |v|e_{[v]},$$
with $[\cdot]:V\to\mP V$ the canonical projection.
However, $k{\mP V}\not\simeq V$ if $\dim_kV>1$.
\end{remark}

 Recall our convention from~\ref{creator} concerning right adjoints.
\begin{lemma}\label{LemRefl}
Consider $u\in[\ba, \bC]$ for which the right adjoint $R:\bC\to\PSh\ba$, $M\mapsto\bC(u-,M)$ is fully faithful ($u$ is `dense'). 
\begin{enumerate}[label=(\roman*)]
\item Assume that $\ba$ is additive. For an object $M\in\bC$, consider the (comma) category $\bj_M$ with as objects pairs $(A,f)$ with $A\in\ba$ and $f:u(A)\to M$, and where a morphism $(A,f)\to (B,g)$ is a morphism $a:A\to B$ in $\ba$ such that $u(a),f,g$ yield a commutative diagram. For the obvious forgetful functor $J_M:\bj_M\to \ba$, we have a canonical isomorphism
$$\colim(u\circ J_M)\;\stackrel{\sim}{\to}\; M.$$
\item Consider $k$-linear cocontinuous functors $H_1,H_2:\bC\to\bB$ and a natural transformation $\eta:H_1\to H_2$. If $\eta$ yields an isomorphism $H_1\circ u\stackrel{\sim}{\to} H_2\circ u$ (resp. $\eta_u=0$), then $\eta$ is an isomorphism (resp. zero) too.
\end{enumerate}
\end{lemma}
\begin{proof}
Clearly, $\bj_M$ is equivalent to the category $\el:=\el(R M)$. 
We can thus use $\colim (\Yon\circ J)\stackrel{\sim}{\to} R M$, for the forgetful $J:\el\to \ba$, from Lemma~\ref{LemYon}. Part (i) then follows from the isomorphisms, natural in $N\in\bC$:
\begin{eqnarray*}
\bC(M,N)&\stackrel{\sim}{\to}&\PSh\ba(R M,R N)\;\stackrel{\sim}{\to}\;\lim \PSh\ba(\Yon\circ J,R N)\;\stackrel{\sim}{\to}\;\lim\PSh\ba(\Yon\circ J_M,R N)\\
&\stackrel{\sim}{\to}&\lim \bC(u\circ J_M,N)\;\stackrel{\sim}{\to}\;\bC(\colim u\circ J_M,N).
\end{eqnarray*}
Part (ii) is an immediate application of part (i), since we can replace $\ba$ by its additive envelope. 
\end{proof}

\begin{proof}[Proof of Proposition~\ref{PropUni}]
For a reflective subcategory $\bC$ of $\PSh\ba$, the inclusion $\tI:\bC\to\PSh\ba$ is the right adjoint of $\tZ:\ba\to\bC$ (more precisely of $\tS$). We can thus apply Lemma~\ref{LemRefl}.

Since $\ba$ is essentially small and $\bB$ is cocomplete, for any $k$-linear functor $h:\ba\to\bB$, we have the left Kan extension 
$$L(h):=\Lan_{\tZ}(h)\,:\,\Sh_{\cS}\ba\to\bB$$ along $\tZ$, see \cite[Chapter X]{Mac} or \cite[Chapter~4]{Ke}. By definition, there exists a natural transformation $h\to L(h)\circ\tZ$ which induces an isomorphism
$$\Nat(L(h),G)\;\simeq\;\Nat(h,G\circ \tZ),$$
for any $k$-linear functor $G:\Sh_{\cS}\ba\to\bB$.
In other words, we have an adjunction
\begin{equation}\label{EqAdj}\xymatrix{[\ba,\bB]\ar@/^/[rr]^{L(-)}&&[\Sh_{\cS}\ba,\bB]\ar@/^/[ll]^{-\circ\tZ},}\end{equation}
and we will show it restricts to an equivalence on the requested subcategories.

For a sequence \eqref{EqSeq} in $\cS$, the sequence
$$\bigoplus_\gamma \tZ(Z_\gamma)\;\to\;\bigoplus_\beta \tZ(Y_\beta)\;\to\; \tZ(X)\;\to\; 0$$
is exact in $\Sh_\cS\ba$, by Corollary~\ref{CorTriv}. Since $H\in[\Sh_{\cS}\ba,\bB]_{cc}$ is right exact, it maps the above exact sequence to an exact one. Since $H$ commutes with coproducts, it follows that $h:=H\circ\tZ$ is in $[\ba,\bB]_{\cS}$, which means $-\circ\tZ$ restricts indeed to a functor $[\Sh_{\cS}\ba,\bB]_{cc}\to [\ba,\bB]_{\cS}$.

Now assume that $h\in  [\ba,\bB]_{\cS}$, then we have the functor 
$$\hat{h}:\bB\to\Sh_{\cS}\ba,\; M\mapsto \bB(h-,M).$$
We claim that $(L(h),\hat{h})$ is an adjoint pair, which shows in particular that $L(h)$ is cocontinuous. Indeed, this standard property can be derived from the coend expression
$$\Lan_{\tZ}h(F)\;=\;\int^{X\in\ba}\Sh_{\cS}\ba(\tZ X,F)\odot_kh(X)\;\simeq\;\int^{X\in\ba}F(X)\odot_kh(X),$$
for $F\in\Sh_{\cS}\ba$, see~\cite[Equation~(4.25)]{Ke}.
In the above formula, $N\odot_k B$ for $N\in \Mod$ and $B\in\bB$ stands for the object in $\bB$ with the universal property $\bB(N\odot_k B,-)\simeq\Hom_k(N,\bB(B,-))$.

Hence \eqref{EqAdj} restricts to an adjunction between  $[\ba,\bB]_{\cS}$ and $[\Sh_{\cS}\ba,\bB]_{cc}$. 
The defining morphism $h\to L(h)\circ\tZ$ is actually an isomorphism, as can be seen from the above co-end expression and the co-Yoneda lemma.
In other words, the unit of $L(-)\dashv -\circ \tZ$ is an isomorphism. 

To prove that $L(-)$ and $-\circ\tZ$ are inverses, it is now enough to show that the right adjoint, $-\circ \tZ$, reflects isomorphisms. The latter is precisely Lemma~\ref{LemRefl}(ii).
\end{proof}

\begin{remark}
Consider a cocontinuous functor $H:\Sh(\ba,\cT)\to\bB$, with $\cT$ a topology. Define $h:=H\circ\tZ$. The above proof shows that $H\simeq \Lan_{\tZ}h$ and its right adjoint is given by $M\mapsto \bB(h-,M)$.
\end{remark}

\subsubsection{} Denote by $\Top(\ba)$ the partially ordered (by inclusion) set of (linear) Grothendieck topologies on $\ba$. We interpret $\Top(\ba)$ as a category and we will show that $\Top(\ba)$ is equivalent to category which is an obvious truncation of a more natural $2$-category.

We introduce the category $\Loc(\ba)$. Objects in $\Loc(\ba)$ are creators $u:\ba\to \bC$. Morphisms $u\to u'$ are the isomorphism classes of cocontinuous $k$-linear functors $\bC\to \bC'$ which yield a commutative diagram with $\bC\xleftarrow{u}\ba\xrightarrow{u'}\bC'$. 

\begin{corollary}\label{CorPOEquiv}
The assignment $\cT\mapsto \{\tZ:\ba\to \Sh(\ba,\cT)\}$ yields an equivalence $\Top(\ba)\stackrel{\sim}{\to}\Loc(\ba)$.
\end{corollary}
\begin{proof}
Consider two Grothendieck topologies $\cT_1,\cT_2$, set $\cS_1=\preo(\cT_1)$ and consider $\tZ_i:\ba\to \Sh(\ba,\cT_i)$. By Proposition~\ref{PropUni}, we have an equivalence
$$-\circ \tZ_1:\;[\Sh_{\cS_1}\ba,\Sh(\ba,\cT_2)]\;\stackrel{\sim}{\to}\;[\ba,\Sh(\ba,\cT_2)]_{\cS_1}.$$ 
Functors $F$ in the left category for which there exists an isomorphism $F\circ \tZ_1\simeq\tZ_2$ thus exist (and are then all isomorphic) if and only if $\tZ_2$ maps all sequences in $\cS_1$ to right exact sequences in $\Sh(\ba,\cT_2)$. That condition is a reformulation of $\preo(\cT_1)\subset\pre(\cT_2)$, which we claim is equivalent with $\cT_1\subset\cT_2$. Indeed, by Theorem~\ref{ThmTPT}(iii) and (iv), $\preo(\cT_1)\subset\pre(\cT_2)$ implies $\cT_1\subset\cT_2$. Also by Theorem~\ref{ThmTPT}(iv), $\cT_1\subset\cT_2$ implies $\preo(\cT_1)\subset\preo(\cT_2)$. However, we have $\preo(\cT_2)\subset\pre(\cT_2)$.
 This demonstrates we can extend the assignment in the corollary to a fully faithful functor $\Top(\ba)\to\Loc(\ba)$. The functor is essentially surjective by Theorem~\ref{ThmBQ}.
\end{proof}

\begin{remark}
One can verify that the morphisms in $\Loc(\ba)$ automatically correspond to {\em exact} functors.
\end{remark}


\section{Properties of linear topologies and pretopologies}\label{PropTop}

\subsection{Noetherian topologies}
For the entire subsection, we let $\kappa$ be an infinite regular cardinal. 

\begin{definition}\label{DefNoe}
A Grothendieck topology $\cT$ is {\bf $\kappa$-noetherian} if for every $X\in\ba$, the object $\tZ(X)$ is $\kappa$-compact in $\Sh(\ba,\cT)$.
\end{definition}

It seems difficult to characterise noetherianity directly from the topology itself, see for instance Example~\ref{ExNoe}. In the following proposition we investigate the connection with the generation and presentation of the sieves in the topology. Recall from Lemma~\ref{LemCompact} that a sieve $R$ is $\kappa$-compact in $\PSh\ba$ if and only if it has a $\kappa$-bounded presentation by representable presheaves. Similarly, $R\subset\ba(-,X)$ is $\kappa$-generated in  $\PSh\ba$ if $R=R_v$ for some 
$\kappa$-bounded $v:\amalg_\beta Y_\beta\to X$.

\begin{prop}\label{PropNoe}
\begin{enumerate}[label=(\roman*)]
\item If for every $X\in\ba$ and $R\in \cT(X)$ there exists $\cT(X)\ni R'\subset R$ such that $R'$ is $\kappa$-compact in $\PSh\ba$, then $\cT$ is $\kappa$-noetherian.
\item The following are equivalent:
\begin{enumerate}
\item For every $X\in\ba$ and $R\in \cT(X)$ there exists $\cT(X)\ni R'\subset R$ such that $R'$ is $\kappa$-generated.
\item For every $X\in\ba$, the object $\tZ(X)$ is $\kappa$-generated in $\Sh(\ba,\cT)$.
\end{enumerate}
\end{enumerate}
\end{prop}

\begin{proof}
For part (i), we consider a functor $J:\bj\to \Sh(\ba,\cT)$ from a $\kappa$-filtered category $\bj$. By cocontinuity of $\tS$, the colimit of $J$ in $\Sh(\ba,\cT)$ is given by $\tS(\colim \tI\circ J)$. Together with the adjunction $\tS\dashv\tI$ this shows
$$\Sh(\ba,\cT)(\tZ(X),\colim J )\;\simeq\; \Nat(\Yon(X),\tI\circ \tS(\colim \tI\circ J)).$$

Recall that $\tI\circ\tS\simeq\Sigma\circ\Sigma$. We can rewrite the direct limit expression for $\Sigma$ in~\eqref{SigmaDL} by restricting to the cofinal (by assumption) subset of $\kappa$-compact sieves in the topology. 
Hence $\Sigma$ commutes with $\kappa$-filtered colimits and $\tZ(X)$ inherits $\kappa$-compactness from $\Yon(X)$.

That (ii)(a) implies (ii)(b) follows as in the proof of part (i), by investigating $\kappa$-filtered colimits of monomorphisms. We need to use additionally that $\tI$ and $\Sigma$ are left exact, see \ref{sheafif}(i).

Now assume that $\tZ(X)$ is $\kappa$-generated and consider $R\in\cT(X)$. Write $R\subset \Yon(X)$ as the image of $\oplus_{b}\Yon(J_b)\to \Yon(X)$ for a collection of morphisms $J_b\to X$. By Theorem~\ref{ThmBQ} and the fact that $\tS$ is cocontinuous, $\oplus_{b}\tZ(J_b)\to \tZ(X)$ is an epimorphism in $\Sh(\ba,\cT)$. By Lemma~\ref{LemCompact}(i), there exists a subset $\{Y_\beta\to X\}$ of $\{J_b\to X\}$ of cardinality $<\kappa$
such that $\oplus_{\beta}\tZ(Y_\beta)\to\tZ(X)$ remains an epimorphism. Denote by $R'\subset R$ the sieve on $X$ generated by the morphisms $Y_\beta \to X$. Since $\tS$ is exact, and again by Theorem~\ref{ThmBQ}, we have $R'\in \cT(X)$.
\end{proof}

It is impossible to improve \ref{PropNoe}(i) to the same form of \ref{PropNoe}(ii), as the following example shows. 

\begin{example}\label{ExNoe}For a commutative ring $K$, and $x\in K$, we have the pretopology on $\underline{K}$
$$\cS\;=\;\{\bullet\xrightarrow{0}\bullet\xrightarrow{x}\bullet\}.$$
The corresponding topology $\cT:=\topp(\cS)$ consists of all ideals in $K$ which contain $x^i$ for some $i\in\mN$.
By Theorem~\ref{ThmNoe} below, $\cT$ is noetherian (see also~\ref{ExamCommut}(i)). However, the premise of Proposition~\ref{PropNoe}(i) is not always satisfied. Indeed, let $k$ be a field and let $K$ be the quotient of the polynomial ring $k[x_i\,|\,i\in\mN]$ in countably many variables by the ideal $\langle x_ix_j\,|\, i\not= j\rangle$, and set $x:=x_0\in K$. No non-zero ideal contained in $Kx_0\in\cT$ is finitely presented.

\end{example}

\begin{theorem}\label{ThmNoe}
The following conditions are equivalent on a topology $\cT$:
\begin{enumerate}[label=(\roman*)] 
\item $\cT$ is $\kappa$-noetherian;
\item $\cT=\topp(\cS)$ for a $\kappa$-bounded pretopology $\cS$;
\item For every functor $J:\bj\to\Sh(\ba,\cT)$ from a $\kappa$-filtered category $\bj$, the canonical morphism $ \colim(\tI\circ J)\to\tI(\colim J)$ is an isomorphism.
\end{enumerate}
\end{theorem}

We start the proof with a lemma.
\begin{lemma}\label{LemNoe}
Consider a $\kappa$-noetherian topology $\cT$. Denote by $\pre_\kappa\cT\subset\pre\cT$ the subclass of all $\kappa$-bounded sequences ($\kappa$-bounded sequences \eqref{EqSeq} which become right exact in $\Sh(\ba,\cT)$). Then $\pre_\kappa\cT$ is a pretopology with $\topp(\pre_\kappa\cT)=\cT$.
\end{lemma}
\begin{proof}

By applying Lemma~\ref{LemCompact}(ii) for $\Sh(\ba,\cT)$ to any sequence in $\pre(\cT)$, we can cut it down to a $\kappa$-bounded sequence which is still in $\pre(\cT)$.

Since $\pre(\cT)$ is a pretopology, it follows that for $\pre_\kappa(\cT)$ the analogues of (PTa) and (PTb) are satisfied where we only require $q'$ to be in $\Co_A(\pre\cT)=\widetilde{\Co}_A(\pre\cT)$. By the first paragraph, we can replace $q'$ by an element of $\Co(\pre_\kappa\cT)$, which yields the requested commutative diagram. That $\topp(\pre_\kappa(\cT))=\cT$ follows from Corollary~\ref{Cor227}(ii) and the first paragraph. 
\end{proof}

\begin{proof}[Proof of Theorem~\ref{ThmNoe}]
Lemma~\ref{LemNoe} shows that (i) implies (ii). By Theorem~\ref{ThmSheaves}, that (ii) implies (iii) is a special case of Lemma~\ref{LemReflComp}(i). By the same reason, that (iii) implies (i) is already observed in the proof of Lemma~\ref{LemReflComp}.
\end{proof}

\begin{corollary}\label{CorNoe}
For a family $\{\cT_i\,|\,i\in I\}$ of $\kappa$-noetherian topologies, the topology $\vee_{i\in I}\cT_i$ is also $\kappa$-noetherian.
\end{corollary}
\begin{proof}
By Theorem~\ref{ThmTPT}(v), this follows from Theorem~\ref{ThmNoe}
\end{proof}

\begin{corollary}
Consider a $\kappa$-compact creator $u:\ba\to\bC$ of a Grothendieck category $\bC$. Then $\bC$ is equivalent to the full subcategory of $F\in\PSh\ba$ for which
$$0\;\to\; F(X)\;\to\;\prod_\beta F(Y_\beta)\;\to \; \prod_\gamma F(Z_\gamma)$$
is exact for every $\kappa$-bounded sequence \eqref{EqSeq} for which the following is exact:
$$\bigoplus_\gamma u(Z_\gamma)\;\to\;\bigoplus_\beta u(Y_\beta)\;\to\; u(X)\;\to\;0.$$
\end{corollary}
\begin{proof}
This is an immediate application of Lemma~\ref{LemNoe} and Theorem~\ref{ThmSheaves}.
\end{proof}
In case $\kappa=\aleph_0$ and $\ba$ is additive, we have the following obvious simplification.

\begin{lemma}\label{Lem2bound}
Assume that $\ba$ is additive. 
\begin{enumerate}[label=(\roman*)]
\item For a topology $\cT$ on $\ba$, the following are equivalent:
\begin{enumerate}
\item $\cT$ is noetherian;
\item  $\cT=\topp(\cS)$ for a pretopology $\cS$ consisting of sequences $Z\to Y\to X$ in $\ba$;
\item $\pre_2\cT$ is a pretopology and $\cT=\topp(\pre_2\cT)$.
\end{enumerate}
\item Consider a compact creator $u:\ba\to\bC$ of a Grothendieck category $\bC$. Then $\bC$ is equivalent to the category of functors $F:\ba^{\op}\to \Mod$ such that $0\to F(X)\to F(Y)\to F(Z)$ is exact for every sequence $Z\to Y\to X$ in $\ba$ for which $u(Z)\to u(Y)\to u(X)\to0$ is exact in $\bC$.
\end{enumerate}
\end{lemma}

In case $u$ is fully faithful, Lemma~\ref{Lem2bound}(ii) was proved in \cite[Theorem~3.3.1]{Sch2}.


\subsection{Subcanonical topologies}

\begin{definition}\label{defsub}
A topology $\cT$ is {\bf subcanonical} if every representable presheaf is a $\cT$-sheaf.
\end{definition}

\begin{theorem}\label{ThmSub}
The following are equivalent.
\begin{enumerate}[label=(\alph*)]
\item $\cT$ is subcanonical.
\item $\cT=\topp(\cS)$ for some pretopology $\cS$ comprising only right exact sequences.
\item If $\cT=\topp(\cS)$ for some pretopology $\cS$, then $\cS$ consists of right exact sequences.
\item $\pre(\cT)$ consists of right exact sequences.
\item $\tZ:\ba\to\Sh(\ba,\cT)$ is fully faithful.
\end{enumerate}
\end{theorem}
\begin{proof}
By Corollary~\ref{CorTriv}, we have $\cS\subset\pre(\cT)$ when $\cT=\topp(\cS)$. Hence (d) implies (c). That (c) implies (b) is obvious. That (b) implies (a) follows from Theorem~\ref{ThmSheaves}.
We have $\tI\circ \tZ\simeq\Yon$ when the image of $\Yon$ consists of sheaves, see~\ref{sheafif}(ii), so (a) implies (e).
To prove that (e) implies (d), consider a sequence \eqref{EqSeq} such that 
$$\bigoplus_\gamma\tZ(Z_\gamma)\to \bigoplus_\beta\tZ(Y_\beta)\to\tZ(X)\to0$$
is exact in $\Sh(\ba,\cT)$. This means in particular that application of $\Sh(\ba,\cT)(-,\tZ(A))$ yields an exact sequence for each $A\in\ba$. Since $\tZ$ is fully faithful, we find that
$$0\to \ba(X,A)\to\prod_\beta\ba(Y_\beta,A)\to\prod_\gamma\ba(Z_\gamma,A)$$
is indeed exact. 
\end{proof}

\begin{corollary}\label{CorCan}For a family $\{\cT_i\,|\,i\in I\}$ of subcanonical topologies, the topology $\vee_{i\in I}\cT_i$ is also subcanonical.

\end{corollary}
\begin{proof}
By Theorem~\ref{ThmTPT}(v), this follows from Theorem~\ref{ThmSub}
\end{proof}


\subsection{The canonical topology}

\begin{theorem}\label{ThmCanTop}
\begin{enumerate}[label=(\roman*)]
\item There exists a unique finest subcanonical topology on $\ba$, the {\bf canonical} topology. It is given by $\vee_i\cT_i$, for $\cT_i$ ranging over all subcanonical topologies. 
\item There exists a unique finest subcanonical noetherian topology $\ba$, the {\bf canonical noetherian} topology. It is given by $\vee_i\cT_i$, for $\cT_i$ ranging over all subcanonical noetherian topologies. \end{enumerate}
\end{theorem}
\begin{proof}
This is a direct consequence of Corollories~\ref{CorNoe} and~\ref{CorCan}.

\end{proof}


\begin{example}\label{ExInd}
If $\ba$ is abelian, the canonical noetherian topology $\cT$ is the one for which $\Ind\ba\simeq\Sh(\ba,\cT)$. This is a direct consequence of Theorems~\ref{ThmNoe} and~\ref{ThmSub}, see the following theorem for a complete proof of a generalisation.
\end{example}


\begin{theorem}\label{ThmTensorEnv}
Let $\ba$ be an essentially small full additive subcategory of a $k$-linear Grothendieck category $\bC$ such that $\ba$ is a compact generator and every compact object in $\bC$ is a subobject of one in $\ba$.
Then the following is true:
\begin{enumerate}[label=(\roman*)]
\item The class $\cS$ of all right exact sequences
$Z\to Y\to X$ in $\ba$ (that is sequences such that $X=\coker(Z\to Y)$ )
 constitutes a pretopology on $\ba$.
 \item The topology $\topp(\cS)$ is the canonical noetherian topology on $\ba$.
 \item We have $\bC\simeq\Sh(\ba,\topp(\cS))$ and $\bC$ is equivalent to the category of functors $F:\ba^{\op}\to \Mod$, for which $F(\coker f)\stackrel{\sim}{\to}\ker F(f)$ for every morphism $f$ which has a cokernel.
\end{enumerate}
\end{theorem}
\begin{proof}
By Theorems~\ref{ThmGP} and~\ref{ThmBQ}, $\bC\simeq\Sh(\ba,\cT)$ for some Grothendieck topology $\cT$. By Theorem~\ref{ThmSub}, $\pre(\cT)$ consists of formal right exact sequences. Hence $\pre_2\cT\subset\pre\cT$ consists of right exact sequences in $\ba$ and by Lemma~\ref{Lem2bound}(i), $\pre_2\cT$ is a pretopology and $\cT=\topp(\pre_2\cT)$.

Part (i) thus follows if $\pre_2\cT$ comprises {\em all} right exact sequences in $\ba$. Consider therefore an exact sequence $Z\to Y\to X\to 0$ in $\ba$. We claim it is also exact in $\bC$. Since $\bC$ is locally finitely presentable, every object is a filtered colimit of compact objects. Since filtered colimits of short exact sequences in $\Mod$ are exact, to prove the claim it thus suffices to prove that 
$$0\to \bC(X,V)\to\bC(Y,V)\to\bC(Z,V)$$
is exact, for every compact object $V\in \bC$. By taking a copresentation of $V$ by objects in $\ba$ (which exists by assumption and the fact that a quotient of two compact objects is again compact), it follows that the above sequence is exact, which proves the claim.

For part (ii), we need to show that $\topp(\cS)$ contains every noetherian subcanonical topology. Clearly, $\cS$ contains every pretopology consisting of $2$-bounded right exact sequences. Since $\topp$ is inclusion preserving, the claim follows from Lemma~\ref{Lem2bound}(i) and Theorem~\ref{ThmSub}.

Part (iii) follows from the first paragraph and Theorem~\ref{ThmSheaves}.
\end{proof}

\subsection{Some results on non-subcanonical topologies}

\begin{prop}\label{PropKer}
For a pretopology $\cS$, set $\cT=\topp(\cS)$. For $X,X'\in\ba$ and $\tZ:\ba\to\Sh(\ba,\cT)$, the kernel of $\ba(X,X')\to \Sh(\ba,\cT)(\tZ X,\tZ X')$ is given by
$$J(X,X')\;=\; \bigcup_{r:\amalg_\alpha V_\alpha\to X\,\in\, \widetilde{\Co}(\cS)}\ker\left(\ba(X,X')\to\prod_\alpha\ba(V_\alpha,X')\right).$$
In particular, $\tZ$ is faithful if and only if every $q:\amalg_\beta Y_\beta\to X$ in $\Co(\cS)$ is epimorphic.
\end{prop}
\begin{proof}
We use the adjoint pair $(\tS,\tI)$ so that the composite
$$\ba(X,X')\;\to\; \Sh(\ba,\cT)(\tZ X,\tZ X')\;\xrightarrow{\sim}\; (\tI\tS \ba(-,X'))(X)$$
is precisely the double evaluation of the unit $\id\to\tI \tS$ at $\Yon(X')$ and $X$. 
By \ref{sheafif}(iii) and (iv), $\sigma_{\Sigma F} :\Sigma F\to \Sigma\Sigma F$ is a monomorphism for every presheaf $F$. The kernel $J(X,X')$ is therefore the kernel of
$$\ba(X,X')\;\to\; (\Sigma \ba(-,X'))(X)=\varinjlim_{R\in\cT(X)}\Nat(R,\Yon(X')).$$
 By definition, every $R\in\cT(X)$ contains some $R_r\in\cT(X)$ with $r:\amalg V_\alpha\to X$ in $\widetilde{\Co}(\cS)$. Moreover, by the exact sequence~\eqref{NatPres} applied to $F=\Yon(X')$, the kernel of $\ba(X,X')\to \Nat(R_r,\Yon(X'))$ is the kernel of $\ba(X,X')\to\prod_\alpha \ba(V_\alpha, X')$. This leads to the description of the kernel.
\end{proof}

\begin{remark}
For the specific case $\cS=\pre(\cT)$, Proposition~\ref{PropKer} is~\cite[Lemma~3.6]{Lo}.
\end{remark}

\begin{lemma}\label{lemlim}
Consider a pretopology $\cS$ for which every $q\in \Co(\cS)$ is an epimorphism. Assume we have $X,X'\in \ba$ such that $\Co_X(\cS)=\widetilde{\Co}_X(\cS)$. If for every $\amalg_\gamma Z_\gamma\to \amalg_\beta Y_\beta\to X$ in $\cS$, the sequence
$$0\to \ba(X,X')\to\prod_\beta \ba(Y_\beta, X')\to\prod_\gamma\ba(Z_\gamma,X') $$
is exact, then $\tZ$ induces an isomorphism $\ba(X,X')\to \Sh_{\cS}(\tZ(X),\tZ(X')).$
\end{lemma}
\begin{proof}
As in the proof of Proposition~\ref{PropKer}, we have to prove that sheafification induces an isomorphism
$$\Yon(X')(X)\;\to\; (\tI\tS \Yon(X'))(X).$$
Moreover, by Lemma~\ref{LemSep}, the presheaf $\Yon(X')$ is separated. 

We continue by considering an arbitrary separated presheaf $F$, for which we have $\Sigma F\simeq \tI\tS F$ by \ref{sheafif}(ii) and (v). By the definition of $\topp(\cS)$ we also find that we can restrict the limit in \eqref{SigmaDL} to the sieves in $\topp(\cS)(X)$ of the form $R_q, q\in \Co_X(\cS)$:
$$\Sigma F(X)=\varinjlim \Nat(R_q,F).$$
Now, by exact sequence~\eqref{NatPres} and Lemma~\ref{LemChiant}(i), we have an exact sequence
$$0\to \Nat(R_q, F)\to \prod_\beta F(Y_\beta)\to \prod_\gamma F(Z_\gamma),$$
which shows that, for $F=\Yon(X')$, the morphism $F(X)\to\Nat(R_q,F)$ is actually an isomorphism, which completes the proof.
\end{proof}

\subsection{Monoidal topologies}\label{MonTop}

By a `monoidal category' we will always understand a $k$-linear and monoidal category for which the tensor product is $k$-linear in each variable.

\subsubsection{Biclosed categories}A monoidal category $(\bB,\otimes,\unit)$ is {\bf biclosed} if for every $X\in\bB$ the endofunctors $X\otimes-$ and $-\otimes X$ of $\bB$ have right adjoints $[X,-]_l$ and $[X,-]_r$. We will abbreviate `biclosed monoidal' to `biclosed'. If $\bB$ is a Grothendieck category, then $(\bB,\otimes,\unit)$ is biclosed if and only if the tensor product is cocontinuous in each variable. 

\subsubsection{Day convolution}\label{DayConv} Let $(\ba,\otimes,\unit)$ be an essentially small monoidal category. Then $\PSh\ba$ is a biclosed Grothendieck category for Day convolution, which defines the tensor product via left Kan extensions, see e.g.~\cite{IK, Sch2}. Concretely, the tensor product of $F,G\in\PSh\ba$ is
$$F\ast G\;=\; \int^{X,Y\in\ba}F(X)\otimes_k G(Y)\otimes_k \ba(-,X\otimes Y)$$
and for a third $H\in\PSh\ba$ we have
$$[F,H]_l\;=\;\int_{X\in\ba}\Hom_k(F(X),H(X\otimes -)),\;\mbox{ or }\; [F,H]_l(Z)=\Nat(F,H(-\otimes Z)),$$
and
$$[G,H]_r\;=\;\int_{Y\in\ba}\Hom_k(G(Y),H(-\otimes Y)),\;\mbox{ or }\; [F,H]_r(Z)=\Nat(G,H(Z\otimes -)).$$
The Yoneda embedding $\Yon:\ba\to\PSh\ba$ is canonically monoidal and, by \cite[Theorem~5.1]{IK}, composition with $\Yon$ yields an equivalence
\begin{equation}\label{UniIK}
[\PSh\ba,\bB]^{\otimes}_{cc}\;\stackrel{\sim}{\to}\; [\ba,\bB]^{\otimes},
\end{equation}
for any $k$-linear cocomplete monoidal category with cocontinuous tensor product $\bB$. Here and below, for any  category $C$ of functors between monoidal categories, we denote by $C^\otimes$ the category of monoidal functors (and monoidal natural transformations) for which the underlying functors lie in $C$.

\subsubsection{}Now consider a biclosed Grothendieck category $\bC$ with an essentially small full monoidal subcategory $\bc\subset\bC$ which is also a generator. Clearly, we can always construct such a $\bc$.
By \eqref{UniIK}, the sheafification $\tS:\PSh\bc\to \bC$ is monoidal. More generally, for $\ba$ monoidal, we call a localisation $\bC$ of $\PSh\ba$,  for which $\bC$ and $\tS:\PSh\ba\to\bC$ (or equivalently $\tZ:\ba\to\bC$) have monoidal structures, a {\bf monoidal localisation}.
Theorem~\ref{ThmGP} and the above show that all biclosed Grothendieck categories are such monoidal localisations of presheaf categories. 

This motivates the notion of `monoidal topologies' introduced below. 

\subsubsection{}\label{A*} For a formal sequence $\sigma$ as in \eqref{EqSeq} and $A\in\ba$, we denote by $A\otimes \sigma$ the obvious formal sequence
$$\amalg_\gamma(A\otimes Z_\gamma)\;\to\; \amalg_\beta (A\otimes Y_\beta)\;\to\; A\otimes X,$$
and similarly for $\sigma\otimes A$. For a sieve $R\subset\ba(-,X)$ and $A\in \ba$, we denote by $A\otimes R$ the sieve on $A\otimes X$ for which $(A\otimes R)(V)$ is generated by all composites $V\to A\otimes Y\xrightarrow{A\otimes f} A\otimes X$
with $f\in R(Y)$ and $Y\in\ba$. We define $R\otimes A$ similarly. We can observe that $A\otimes R$ is the image in $\PSh\ba$ of the canonical morphism
$$\Yon(A)\ast  R\;\to\; \Yon(A)\ast \Yon(X)\simeq \Yon(A\otimes X),$$
which, by cocontinuity of $\ast$, is indeed the image of $\oplus_\beta \Yon(A\otimes Y_\beta)\to \Yon(A\otimes X)$ for a choice of generators $Y_\beta\to X$ of $R$.

\begin{theorem}\label{ThmBiclosed}
Let $(\ba,\otimes,\unit)$ be an essentially small monoidal category. The following conditions are equivalent on a Grothendieck topology $\cT$ on $\ba$.
\begin{enumerate}[label=(\alph*)]
\item The class $\pre(\cT)$ is closed under the operations $A\otimes-$ and $-\otimes A$, for every $A\in\ba$.
\item We have $\cT=\topp(\cS)$ for some pretopology $\cS$ closed under $A\otimes-$ and $-\otimes A$, for every $A\in\ba$.
\item For every $X\in\ba$ and $R\in\cT(X)$, we have $R\otimes A\in \cT(X\otimes A)$ and $A\otimes R\in \cT(A\otimes X)$.
\item For every $F\in \Sh(\ba,\cT)$ and $A\in\ba$, the presheaves $F(A\otimes-)$ and $F(-\otimes A)$ are $\cT$-sheaves.
\item There exists a (automatically essentially unique) biclosed structure on $\Sh(\ba,\cT)$ for which $\tZ$ admits a monoidal structure.
\end{enumerate}
A topology which satisfies one of these conditions is {\bf monoidal}.
\end{theorem}
\begin{proof} We start by proving the cycle $a\Rightarrow b \Rightarrow d\Rightarrow e\Rightarrow a$.
That (a) implies (b) follows from Theorem~\ref{ThmTPT}(iii). That (b) implies (d) is a consequence of Theorem~\ref{ThmSheaves}.

That (d) implies (e) is a consequence of Day's reflection theorem. Concretely, by \cite[Theorem~1.2(ii)]{Day} applied to the generating subcategory $\ba\subset\PSh\ba$ and the localisation $\Sh(\ba,\cT)\subset\PSh\ba$, the condition that $[\Yon(A),F]_l\simeq F(A\otimes-)$ and $[\Yon(A),F]_r\simeq F(-\otimes A)$ be $\cT$-sheaves, for every $\cT$-sheaf $F$ and $A\in \ba$, implies there exists a biclosed structure on $\Sh(\ba,\cT)$ for which $\tS:\PSh\ba\to\Sh(\ba,\cT)$ is monoidal. Such a monoidal structure must be essentially unique. By equivalence~\eqref{UniIK}, uniqueness of this monoidal structure is still imposed by demanding that $\tZ$ has a monoidal structure. In principle, \cite[Theorem~1.2(ii)]{Day} is only concerned with {\em non-enriched} closed {\em symmetric} monoidal categories. However, as also pointed out in \cite[\S 0]{Day}, the methods extend trivially to enriched and biclosed monoidal categories.

That (e) implies (a) follows from the assumption that the tensor product on $\Sh(\ba,\cT)$ is cocontinuous in each variable. 


To prove that (e) implies (c), we use the characterisation the of sieves in a topology from Theorem~\ref{ThmBQ}.
By~\ref{A*}, we can decompose $\Yon(A)\ast i$, with $i$ the inclusion $R\subset \Yon(X)$, as follows:
\begin{equation}\label{compA}\Yon(A)\ast i\;:\; \Yon(A)\ast R \tto A\otimes R \hookrightarrow \Yon(A\otimes X) \stackrel{\sim}{\to} \Yon(A)\ast\Yon(X).\end{equation}
Under assumption (e), $\tS$ is monoidal by \eqref{UniIK}. That $\tS(\Yon(A)\ast i)$ is an isomorphism thus follows since $\tS(i)$ is an isomorphism. Since $\tS$ sends epimorphisms to epimorphisms, it follows that $\tS$ sends $A\otimes R \hookrightarrow \Yon(A\otimes X)$ to an isomorphism, so $A\otimes R\in \cT(A\otimes X)$. The same observation for $R\otimes A$ shows that (c) follows.

Finally, we prove that (c) implies (d). Let $F$ be a $\cT$-sheaf. We will only prove that $F(A\otimes X)\to \Nat(R,F(A\otimes-))$ is an isomorphism for all $X,A\in\ba$ and $R\in \cT(X)$, as the case $-\otimes A$ is done identically.
By adjunction, it suffices to prove that $F(A\otimes X)\to \Nat(\Yon(A)\ast R,F)$ is an isomorphism. The composite
$$F(A\otimes X)\to \Nat(\Yon(A)\ast R,F)\hookrightarrow \Nat(A\otimes R ,F),$$
where the monomorphism is induced from the epimorphism in \eqref{compA}, is an isomorphism by assumption (c). Hence the left morphism in the composite is an isomorphism too.
This concludes the proof.
\end{proof}

\begin{theorem}\label{ThmUniMon}
Consider an essentially small monoidal category $(\ba,\otimes,\unit)$ with monoidal topology $\cT$ and a pretopology (not necessarily closed under tensor product) $\cS$ with $\cT=\topp(\cS)$. Consider also a subclass $\cS_0\subset \cS$ such that each $\gamma\in\cS$ is of the form $A\otimes\gamma'\otimes B$ for some $A,B\in\ba$ and $\gamma'\in\cS_0$.
 Let $\bB$ be a $k$-linear cocomplete monoidal category with cocontinuous tensor product.
Composition with the monoidal functor $\tZ:\ba\to\Sh(\ba,\cT)$ yields an equivalence
$$[\Sh(\ba,\cT),\bB]_{cc}^{\otimes}\;\stackrel{\sim}{\to}\; [\ba,\bB]_{\cS_0}^{\otimes}$$ 
\end{theorem}
\begin{proof}
As the tensor product in $\bB$ is cocontinuous, it follows that $[\ba,\bB]_{\cS_0}^{\otimes}=[\ba,\bB]_{\cS}^{\otimes}$. The statement thus becomes the monoidal version of Proposition~\ref{PropUni} and is a standard consequence of the latter, see e.g. \cite[\S 3]{Sch2}. We sketch an argument below.

As $\tZ$ is monoidal, the functor $[\Sh(\ba,\cT),\bB]_{cc}^{\otimes}\to [\ba,\bB]_{\cS}^{\otimes}$ is well-defined and also faithfulness is inherited from \ref{PropUni}. If for a natural transformation $\eta: F\to G$ for $F,G\in [\Sh(\ba,\cT),\bB]_{cc}^{\otimes}$, we have that $\eta_{\tZ}:F\circ \tZ\to G\circ \tZ$ is monoidal, $\eta$ is itself monoidal, by cocontinuity of $F,G,-\otimes-$ and an application of Lemma~\ref{LemRefl}(ii). Consequently, $[\Sh(\ba,\cT),\bB]_{cc}^{\otimes}\to [\ba,\bB]_{\cS}^{\otimes}$ is full.

Finally, that any monoidal structure on $F\circ \tZ$ for $F\in [\Sh(\ba,\cT),\bB]_{cc}$ extends to one on $F$, follows from \eqref{UniIK}. An alternative approach uses the results of Section~\ref{SecKelly} as follows.  A monoidal structure on $F\circ \tZ$ yields a natural isomorphism between composite functors
$$
\xymatrix{
\ba\otimes\ba\ar[r]\ar[dr]& \ba\ar[dr]\\
& \bB\otimes \bB\ar[r]&\bB,
}
$$
where the horizontal arrows represent the tensor product and diagonal arrows correspond to $F\circ \tZ$. By Proposition~\ref{PropUni} and Theorem~\ref{ThmLRS},
this gives a natural isomorphism between corresponding cocontinuous functors $\bA\boxtimes\bA\rightrightarrows \bB$, for $\bA=\Sh(\ba,\cT)$. The universal property of $\bA\boxtimes\bA$ thus yields a binatural isomorphism of bifunctors $\bA\times\bA\rightrightarrows \bB$. One can then verify that this isomorphism represents the sought after monoidal structure on $F$.
\end{proof}

When a category is braided monoidal, the conditions of being left or right closed are clearly equivalent, so we simply speak of `closed categories'. The following theorems are immediate analogues of the above results.
\begin{theorem}\label{ThmSym}
Let $(\ba,\otimes,\unit,c)$ be an essentially small braided monoidal $k$-linear category. A Grothendieck topology $\cT$ on $\ba$ is monoidal if and only if one of the following is satisfied.
\begin{enumerate}[label=(\alph*)]
\item The class $\pre(\cT)$ is closed under the operation $A\otimes-$, for every $A\in\ba$.
\item We have $\cT=\topp(\cS)$ for some pretopology $\cS$ closed under $A\otimes-$, for every $A\in\ba$.
\item For every $X\in\ba$ and $R\in\cT(X)$, we have $R\otimes A\in \cT(X\otimes A)$.
\item For every $F\in \Sh(\ba,\cT)$ and $A\in\ba$, the presheaf $F(A\otimes-)$ is a $\cT$-sheaf.
\item There exists a closed structure on $\Sh(\ba,\cT)$ for which $\tZ$ admits a braided monoidal structure.
\end{enumerate}

\end{theorem}

\begin{theorem}
Let $(\ba,\otimes,\unit,c)$ be an essentially small braided monoidal category and $\bB$ a $k$-linear cocomplete braided monoidal category with cocontinuous tensor product.
Consider a class $\cS_0$ of formal sequences \eqref{EqSeq} in $\ba$ such that the class $\cS$ of sequences of the form $A\otimes\gamma$ with $A\in\ba$ and $\gamma\in\cS_0$ is a pretopology. Set $\cT=\topp\cS$. Composition with the braided monoidal functor $\tZ:\ba\to\Sh(\ba,\cT)$ yields an equivalence
$$[\Sh(\ba,\cT),\bB]_{cc}^{\otimes,b}\;\stackrel{\sim}{\to}\; [\ba,\bB]_{\cS_0}^{\otimes,b}$$ 
\end{theorem}

\subsection{Kelly product}\label{SecKelly}

We show that the Kelly tensor product, as defined below, of two Grothendieck categories as cocomplete categories is again a Grothendieck category. This recovers the notion of `tensor products of Grothendieck categories' from \cite{LRS}. 
\subsubsection{} 

Kelly proved, see \cite[\S 6.5]{Ke} for a more general result, that there is a canonical notion of a tensor product on the class of $\kappa$-cocomplete $k$-linear categories (which always exists). We will employ the term `Kelly tensor product' for the same notion applied to the class of cocomplete $k$-linear categories (which contrary to {\it loc. cit.} we do not prove exists in general).

Consider cocomplete $k$-linear categories $\bA,\bB$ and $\bC$ and a bifunctor $\bA\times\bB\to\bC$, cocontinuous and $k$-linear in each variable. This data is the Kelly tensor product of $\bA$ and $\bB$ if, for every cocomplete $k$-linear category $\bD$, it yields an equivalence between $[\bC,\bD]_{cc}$ and the category of functors $\bA\times\bB\to\bD$ which are $k$-linear and cocontinuous in each variable. We identify the latter category of functors with $[\bA,[\bB,\bD]_{cc}]_{cc}$. Here, and below, we freely use that colimits in functor categories are computed pointwise in the target category.

For two $k$-linear categories $\ba,\bb$, we denote their ordinary tensor product over $k$ by $\ba\otimes\bb=\ba\otimes_k\bb$. Its objects are pairs $(X,Y)$ with $X\in\ba$ and $Y\in\bb$ and 
$$\ba\otimes\bb((X,Y),(X',Y'))\;:=\;\ba(X,X')\otimes_k\bb(Y,Y').$$ For any $k$-linear category $\bd$ we have an equivalence
\begin{equation}\label{EqObv}[\ba\otimes\bb,\bd]\;\simeq\;[\ba,[\bb,\bd]].\end{equation}

The following theorem mainly recovers results from \cite{LRS}. 
\begin{theorem}\label{ThmLRS}
For $k$-linear Grothendieck categories $\bA$ and $\bB$, choose realisations as $\bA\simeq\Sh_{\cS_1}\ba$ and $\bB\simeq\Sh_{\cS_2}\bb$. Denote by $\bA\boxtimes\bB=\bA\boxtimes_k\bB$ the subcategory of $\PSh(\ba\otimes\bb)$ of functors
$F$ for which the sequences
\begin{eqnarray*}
&&0\to F(X,A)\to \prod_\beta F(Y_\beta, A)\to\prod_\gamma F(Z_\gamma,A),\\
&&0\to F(X,A)\to \prod_\beta F(X,B_\beta)\to\prod_\gamma F(X,C_\gamma)
\end{eqnarray*}
are exact, for all $(X,A)\in\ba\otimes\bb$, all $\amalg_\gamma Z_\gamma\to\amalg_\beta Y_\beta\to X$ in $\cS_1$ and $\amalg_\gamma C_\gamma\to\amalg_\beta B_\beta\to A$ in $\cS_2$. 
\begin{enumerate}[label=(\roman*)]
\item $\bA\boxtimes\bB$ is a localisation of $\PSh(\ba\otimes\bb)$, so in particular it
is a Grothendieck category.  
\item $\bA\boxtimes\bB$ is the Kelly product of $\bA$ and $\bB$, so in particular independent of our choices of $(\ba,\cS_1)$ and $(\bb,\cS_2)$.
\item  If $\ba\to\bA$ and $\bb\to\bB$ are $\kappa$-compact creators for an infinite regular cardinality $\kappa$, then the same is true for $\ba\otimes\bb\to\bA\boxtimes\bB$.
\end{enumerate}
\end{theorem}
\begin{proof}
Denote by $\cS_l$ the class of all formal sequences $\amalg_\gamma (Z_\gamma,A)\to\amalg_\beta (Y_\beta,A)\to (X,A)$ in $\ba\otimes \bb$ induced from $\cS_1$; and by $\cS_r$ the class of all sequences $\amalg_\gamma (X,C_\gamma)\to\amalg_\beta (X,B_\beta)\to (X,A)$ induced from $\cS_2$. We also set $\cS=\cS_l\cup\cS_r$.
The equivalence, coming from~\eqref{EqObv},
$$\PSh(\ba\otimes\bb)\;\simeq\; [\ba^{\op},\PSh\bb]$$
allows us to identify $\Sh_{\cS_r}(\ba\otimes\bb)$ with $[\ba^{\op},\bB]$. The inclusion $[\ba^{\op},\bB]\hookrightarrow [\ba^{\op},\PSh\bb]$ has left adjoint given by $F\mapsto  \tS\circ F$ for $F\in[\ba^{\op},\PSh\bb]$, with $\tS$ the sheafification for $\bB$.
Since we can compute limits pointwise it follows that $[\ba^{\op},\bB]$ is a localisation of $[\ba^{\op},\PSh\bb]$. Similarly, $\Sh_{\cS_l}(\ba\otimes\bb)$ is a localisation of $\PSh(\ba\otimes\bb)$. It thus follows from Corollary~\ref{IntGiraud} that 
$$\bA\boxtimes\bB\;=\; \Sh_{\cS}(\ba\otimes\bb)\;=\; \Sh_{\cS_l}(\ba\otimes\bb)\cap\Sh_{\cS_r}(\ba\otimes\bb)$$
is a localisation of $\PSh(\ba\otimes\bb)$, proving part (i).

Let $\bD$ be any $k$-linear cocomplete category. By Proposition~\ref{PropUni} and \eqref{EqObv}, we have
$$[\bA\boxtimes\bB,\bD]_{cc}\;\simeq\; [\ba\otimes\bb,\bD]_{\cS}\;\simeq\;[\ba,[\bb,\bD]_{\cS_2}]_{\cS_1}\;\simeq\;[\ba,[\bB,\bD]_{cc}]_{\cS_1}\;\simeq\;[\bA,[\bB,\bD]_{cc}]_{cc},$$
which proves part (ii).
Part (iii) follows from Theorem~\ref{ThmNoe} and Lemma~\ref{LemReflComp}(iii).
\end{proof}

\begin{remark}
Using Lemma~\ref{LemPTa}(i) one can show that the class $\cS$ in the proof of Theorem~\ref{ThmLRS} satisfies (PTa). If $k$ is a field one can similarly show that (PTb) is satisfied and hence $\cS$ is a pretopology on $\ba\otimes\bb$. For arbitrary commutative ring~$k$, Remark~\ref{RemTPT}(i) shows that $\cT:=\topp(\cS)$ is a topology. In \cite{LRS} it is actually proved that $\bA\boxtimes\bB=\Sh(\ba\otimes\bb,\cT)$. This is thus only clear from our approach when $k$ is a field.
\end{remark}

\begin{example}
As observed in the proof, we have $\PSh\ba\boxtimes\bB\simeq[\ba^{\op},\bB]$.
\end{example}


\section{Presentations of tensor categories}\label{SecTensor}

For the entire section, $k$ is a field. We denote $\Mod$ by $\Vecc$.

\subsection{Definitions and aims}

Part of the motivation for the previous sections comes from applications to the theory of tensor categories (in the sense of \cite{Del90, EGNO}).

\subsubsection{} A monoidal category $(\ba,\otimes,\unit)$ (following the convention in Subsection \ref{MonTop}) is {\bf rigid} if every object $X\in\ba$ has a left dual $X^\ast$ and a right dual ${}^\ast X$, see \cite[\S 2.10]{EGNO}. We say that $\ba$ is `a monoidal category over $k$' if $k\to\ba(\unit,\unit)$ is an isomorphism. For the rest of the paper, {\em we assume that $\ba$ is an additive rigid monoidal category over $k$}.
 An essentially small rigid monoidal category $\bT$ over $k$ is a {\bf tensor category over $k$} if $\bT$ is abelian. Henceforth, when we say `tensor category' it is understood to mean `tensor category over some field extension $K$ of $k$' (possibly $K=k$) and these are considered as $k$-linear categories.

We will use freely that $(-)^\ast$ has an action on morphisms via an isomorphism
$$\ba(X,Y)\,\xrightarrow{\sim}\,\ba(Y^\ast, X^\ast),\quad f\mapsto f^\ast.$$
This thus yields an equivalence $\ba\to\ba^{\op}$ and in particular sends short exact sequences to short exact sequences if $\ba$ is abelian, see~\cite[Proposition~4.2.9]{EGNO}.
 We refer to \cite{Ideals} for the notion of tensor ideals.

\subsubsection{} The ind-completion of a tensor category is a biclosed Grothendieck category, as follows from \cite[\S 7]{Del90} or Theorem~\ref{ThmBiclosed} and Example~\ref{ExInd}. 
 In this section, we investigate the possibilities to present this ind-completion via a monoidal Grothendieck topology on a rigid monoidal category. 
 With slight abuse of notation we will refer to a monoidal functor $\ba\to\bT$ to a tensor category $\bT$ as a {\bf rigid monoidal creator} when the composite $\ba\to\Ind\bT$ is a creator (so when $\Ind\bT\to\PSh\ba$ is a localisation).

We specify the explicit characterisation of creators in \cite[Theorem~1.2]{Lo} to this particular setting. A simplification occurs because, by compactness of the objects in $\bT$ we can restrict to {\em finite} `epimorphic collections', and subsequently by additivity of $\ba$ to single epimorphisms.
\begin{lemma}\label{LemLow}
A functor $u:\ba\to\bT$ to a tensor category $\bT$ is a creator if and only if the following three conditions are satisfied:
\begin{enumerate}[label=(\roman*)] 
\item[(G)] Every object in $\bT$ is a quotient of some $u(X)$, with $X\in\ba$.
\item[(F)] For every morphism $a:u(X)\to u(Y)$ there exists a morphism $q:X'\to X$ such that $u(q)$ is an epimorphism and $a\circ u( q)$ is in the image of $u$.
\item[(FF)] If $u(f)=0$ for  $f\in\ba(X,Y)$, there exists a morphism $q:X'\to X$ such that $u(q)$ is an epimorphism and $f\circ q=0$.
\end{enumerate}
\end{lemma}

\begin{example}
The conditions in Lemma~\ref{LemLow} can be satisfied when $\bT$ is a tensor category over some non-trivial field extension $K/k=\ba(\unit,\unit)$. An example is given for $\ba$ the category `$\bC$' in \cite[\S 2.3.3]{AbEnv}, with char$(k)=0$, and $\bT$ the abelian envelope of $[\mathrm{GL}_0,K]$ from \cite[Theorem~4.2.1]{AbEnv} or \cite{EHS}.
\end{example}
 
\begin{prop}
Assume $\bT$ is a tensor category with full additive rigid monoidal subcategory $\ba$ such that every object in $\bT$ is a quotient of an object in $\ba$.
 Then $\Ind\bT$ is the localisation of $\PSh\ba$ with respect to the canonical noetherian topology on $\ba$, and $\Ind\bT$ is monoidally equivalent to the category of functors $\ba^{\op}\to\Vecc$ which send every cokernel in $\ba$ to the corresponding kernel in $\Vecc$.
\end{prop}
\begin{proof}
Clearly the subcategory $\ba\subset\bT\subset\Ind\bT$ consists of compact objects. Moreover, by duality $(-)^\ast$, every compact object in $\Ind\bT$ (meaning every object in $\bT$) is a subobject of an object in $\ba$. We can then apply Theorem~\ref{ThmTensorEnv}. That the equivalence between $\Ind\bT$ and the category of sheaves on $\ba$ is monoidal follows for instance from Theorem~\ref{ThmUniMon}.
\end{proof}

We also record the following basic observations for future use.

\begin{lemma}\label{LemInj0}
Consider a rigid monoidal creator $h:\ba\to\bT$. 
\begin{enumerate}[label=(\roman*)]
\item For every morphism $f:X\to Y$ in $\ba$, there exists $g:Y\to Z$ such that $g\circ f=0$ and 
$$h(X)\xrightarrow{h(f)}h(Y)\xrightarrow{h(g)}h(Z)$$
is acyclic in $\bT$.
\item If $I$ is an injective object in $\Ind\bT$, then there exists a functor $F:\bj\to\ba$, with $\bj$ filtered, such that $\colim (h\circ F)\simeq I$.
\end{enumerate}
\end{lemma}
\begin{proof}
We start from a morphism $a:B\to A$ in $\ba$ (which will be ${}^\ast f$). Inside $\PSh\ba$, the collection of all morphisms $b:C\to B$ for which $a\circ b=0$ is jointly epimorphic onto the kernel of $a$ in $\PSh\ba$. Applying the exact and cocontinuous sheafification $\PSh\ba\to\Ind\bT$ thus yields an acyclic sequence
$$\bigoplus_{b:C\to B,a\circ b=0}h(C)\,\to\,h(B)\,\to\,h(A).$$
Since the kernel of $h(a)$ is in $\bT$ and thus compact, additivity of $\ba$ implies that we can pick one morphism $b:C\to B$ for which the sequence remains acyclic, see Lemma~\ref{LemCompact}(i).
Using the duality $(-)^\ast$ on $\ba$ then yields the sequence in (i).

For part (ii), by Lemma~\ref{LemRefl}(i), it suffices to show that $\bj_I$ is filtered. This is an immediate application of part (i) since $I$ is injective.
\end{proof}

\subsection{Morphisms to the tensor unit} 

\subsubsection{} We denote by $\cU=\cU(\ba)$ the class of all non-zero morphisms $U\to\unit$ in $\ba$, and by $\cU^0$ the class of all morphisms $U\to\unit$. We consider two potential properties of such $u:U\to \unit$:
\begin{enumerate}[label=(\roman*)] 
\item[(Ep)] The morphism $u:U\to\unit$ is an epimorphism in $\ba$.
\item[(Ex)] The following diagram is a coequaliser in $\ba$:
 $$\xymatrix{U\otimes U\ar@<0.5ex>[r]^-{u\otimes U}\ar@<-0.5ex>[r]_-{U\otimes u}&U\ar[r]^-{u}&\unit}.$$
\end{enumerate}
We will often regard (Ex) via the equivalent formulation that the sequence 
$$\sigma_u:\quad U\otimes U\xrightarrow{u\otimes U-U\otimes u}U\xrightarrow{u}\unit\to 0$$
is exact. We have inclusions $\cU^{ex}\subset\cU^{ep}\subset \cU\subset \cU^0$, for the subclasses of morphisms which satisfy (Ex) and (Ep).

\begin{theorem}\label{ThmNew}
\begin{enumerate}[label=(\roman*)] 
\item For $u\in\cU$, we have $u\in \cU^{ex}$ if and only if $u$ is a strict epimorphism.
\item If $\bT$ is a tensor category, $\cU(\bT)=\cU^{ex}(\bT)$. 
\end{enumerate}
\end{theorem}
\begin{proof}
For part (i), if $u\in\cU^{ex}$ then it is a cokernel, so in particular a strict epimorphism. Now assume that $u:U\to \unit$ is a strict epimorphism and consider $f:U\to V$ in $\ba$ which equalises $U\otimes U\rightrightarrows U$, in other words, $$u\otimes f = f\otimes u.$$
We need to show that $f$ composes to zero with every morphism $g:X\to U$ in $\ba$ with $u\circ g=0$, since the latter then implies that $f$ factors (uniquely) via $u$. By the displayed equation, $u\circ g=0$ implies 
$$0=f\otimes (u\circ g)=u\otimes (f\circ g)=f\circ g\circ (u\otimes X).$$
By adjunction $u\otimes X$ is an epimorphism, so indeed $f\circ g=0$.

Part (ii) follows from part (i). Indeed, since $\unit$ is known to be simple in any tensor category (\cite[Proposition~1.17]{DM}), $\cU=\cU^{ep}$. Since every epimorphism in an abelian category is strict, we have $\cU^{ep}=\cU^{ex}$.
\end{proof}

\begin{remark}\label{RemOX}
For any quasi-coherent sheaf $\cM$ on a scheme $\mX$ with epimorphism $\cM\tto\cO_{\mX}$, the sequence 
$$\cM\otimes_{\cO_{\mX}}\cM\to\cM\to\cO_{\mX}\to 0$$
is exact in $\QCoh\mX$. Indeed, this is clearly the case for affine schemes, since for $\mX=\Spec R$ every epimorphism $M\tto R$ splits. The general claim then follows by taking the stalk at each $x\in\mX$.

\end{remark}

\subsubsection{} For a subclass $\cV\subset\cU^{0}$, we denote by $\bar{\cV}\subset \cU^0$ the closure of $\cV$ under tensor products. Concretely, the elements of $\bar{\cV}$ are the morphisms of the form
$$u_1\otimes\cdots\otimes u_n\;:\; U_1\otimes\cdots \otimes U_n\,\to\,\unit,$$
with $u_i\in\cV$. We then have the sieve $\ann_\cV=\ann_{\bar{\cV}}$ on $\unit$, defined by
$$\ann_{\cV}(X)\;=\;\{u\in \ba(X,\unit)\,|\, u^\ast\circ v=0 \mbox{ for some $v\in\bar{\cV}$}\}$$

We have $\ann_{\cV}=0$ if and only if $\cV\subset\cU^{ep}$. Note also that $\cV_1\subset\cV_2$ implies $\ann_{\cV_1}\subset\ann_{\cV_2}$. 

An {\bf ideal} $J$ in $\ba$ is a subfunctor of $\ba(-,-):\ba^{\op}\times\ba\to \Mod$ and, for $X\in\ba$, we use the notation $J(-,X)$ and $J(X,-)$ correspondingly for the functors on $\ba^{\op}$ and $\ba$. To a functor we can associate an ideal, its {\bf kernel}, comprising of all morphisms sent to zero. For a left tensor ideal $J$, adjunction restricts to an isomorphism
\begin{equation}\label{adjIde}
J(X,Y)\;\xrightarrow{\sim}\; J(Y^\ast\otimes X,\unit).
\end{equation}
In particular, it is well-known, see for instance \cite[Theorem~3.1.1]{Ideals}, that $J\mapsto J(-,\unit)$ yields an inclusion preserving bijection between left tensor ideals in $\ba$ and sieves on $\unit$, where the inverse map sends a sieve $R$ to the ideal $(X,Y)\mapsto R(Y^\ast\otimes X)$. 

We denote by $J_{\cV}$ the unique left tensor ideal in $\ba$ with $J_{\cV}(-,\unit)=\ann_{\cV}$.

\begin{lemma}\label{LemEpiU}
Consider a monoidal creator $h:\ba\to\bC$ with $\bC$ a biclosed Grothendieck category in which $\unit$ is finitely generated.
Let $f: Y\to X$ be a morphism in $\ba$ for which $h(f)$ is an epimorphism. Then there exist $u: U\to\unit$ and $g:U\otimes X\to Y$ in $\ba$ such that $h(u)$ is an epimorphism and $u\otimes X=f\circ g$.
\end{lemma}
\begin{proof}
We know that $h(f\otimes X^\ast)$ is also an epimorphism. We consider $\co_X:\unit\to X\otimes X^\ast$. By Lemma~\ref{LemEpiLoc} and the fact that $\unit$ is finitely generated, there exists $u:U\to\unit$ in $\ba$ such that $h(u)$ is an epimorphism and there exists a commutative diagram
$$\xymatrix{
Y\otimes X^\ast\ar[rr]^{f\otimes X^\ast}&& X\otimes X^\ast\\
U\ar[u]\ar[rr]^{u}&&\unit \ar[u]^{\co_X}.
}$$
The claim now follows by applying the adjunction $-\otimes X\dashv -\otimes X^\ast$.
\end{proof}

\subsection{A set of Grothendieck topologies}

\begin{lemma}\label{PropsTV} Consider a subclass $\cV\subset\cU^{0}$.
\begin{enumerate}[label=(\roman*)] 
\item The class of sequences 
$$\{\sigma_u\otimes X\;:\;\,U\otimes U\otimes X\;\to\; U\otimes X\;\to X\,|\, u\in\cV\},$$
is a pretopology on $\ba$, which we denote by $\cS_{\cV}$. We have $\widetilde{\Co}_{\unit}(\cS_{\cV})=\bar{\cV}$.
 \item The topology $\cT_{\cV}:=\topp(\cS_{\cV})=\cT_{\bar{\cV}}$ is noetherian. The topology is subcanonical if and only if $\cV\subset\cU^{ex}$, and monoidal if $\ba$ is braided.
 \item The functor $\tZ:\ba\to\Sh(\ba,\cT_{\cV})$ is faithful if and only if $\cV\subset\cU^{ep}$ and fully faithful if and only if $\cV\subset\cU^{ex}$. If $\cT_{\cV}$ is monoidal, the kernel of $\tZ$ is $J_{\cV}$.
 \item Every object in $\Sh(\ba,\cT_{\cV})$ is a quotient of a direct sum of objects in the image of $\tZ$. For every morphism $a: \tZ(X)\to \tZ(Y)$ there exists a morphism $q:X'\to X$ in $\ba$ such that $\tZ(q)$ is an epimorphism and $a\circ \tZ(q)$ is in the image of~$\tZ$.
 \end{enumerate}
\end{lemma}
\begin{proof}
For part (i), Condition (PTa) follows from the fact that for any $u\in\cV$ and $f:A\to X$ in $\ba$, there is a commutative diagram
$$\xymatrix{
U\otimes X\ar[rr]^{u\otimes X}&& X\\
U\otimes A\ar@{-->}[rr]^{u\otimes A}\ar@{-->}[u]^{U\otimes f}&& A\ar[u]^f.
}$$
Condition (PTb) follows since for every $u:U\to\unit $ in $\cV$ and $f:A\to U\otimes X$ with $(u\otimes X)\circ f=0$, there is a commutative diagram
$$\xymatrix{
U\otimes U\otimes X\ar[rrr]^-{u\otimes U\otimes X-U\otimes u\otimes X}&&& U\otimes X\\
U\otimes A\ar@{-->}[rrr]^{u\otimes A}\ar@{-->}[u]^{U\otimes f}&&& A\ar[u]^f.
}$$

For part (ii), that $\cT_{\cV}$ is noetherian follows from Lemma~\ref{Lem2bound}(i). That $\cT_{\cV}$ is subcanonical if and only if $\cV\subset\cU^{ex}$ follows from Theorem~\ref{ThmSub}. That $\cT_{\cV}$ is monoidal when $\ba$ is braided follows from Theorem~\ref{ThmSym}.

The first sentence in part (iii) follows from Proposition~\ref{PropKer} and Theorem~\ref{ThmSub}. Now assume that $\cT_{\cV}$ is monoidal and let $J$ denote the tensor ideal in $\ba$ which is the kernel of $\tZ$. By Proposition~\ref{PropKer}, for each $W\in \ba$ we have
$$J(\unit,W)\;=\;\{w:\unit\to W\,|\, w\circ v=0\mbox{ for some $v\in\bar{\cV}$}\}.$$
By~\eqref{adjIde}, this implies $J(U,\unit)=\ann_{\cV}(U)= J_{\cV}(U,\unit)$, so consequently $J=J_{\cV}$.

Part (iv) is just an instance of the general theory of localisations, see \cite[Theorem~1.2]{Lo}.
\end{proof}

\begin{example}\label{ExAE1}
In \cite{AbEnv}, an object $X\in\ba$ is called `strongly faithful' if $\ev_X:X^\ast\otimes X\to\unit$ satisfies (Ex). Note that $\ev_X$ satisfies (Ep) if and only if $X\otimes -$ is faithful.
For $\cV$ the class of $\ev_X$ in $\cU^{ex}$, the topology $\cT_{\cV}$ is studied in \cite[\S 3]{AbEnv}.
\end{example}

\begin{prop}\label{PropTQC}
Consider a noetherian monoidal topology $\cT$ on $\ba$ such that for every epimorphism $N\tto\unit$ in $\Sh(\ba,\cT)$ from a rigid object $N$, the sequence $N\otimes N\to N\to\unit\to0$ is exact. Then $\cT=\cT_{\cV}$, for $\cV$ the class of morphisms $u:U\to\unit$ in $\ba$ for which $\tZ(u)$ is an epimorphism in $\Sh(\ba,\cT)$.
\end{prop}
\begin{proof}
By assumption and since the tensor product in $\Sh(\ba,\cT)$ is right exact, the sequence $\tZ(\sigma_u\otimes X)$ is exact for every $u\in\cV$ and $X\in\ba$. 
That $\cT=\cT_{\cV}$ is an instance of Corollary~\ref{Cor227}(ii), by application of Lemma~\ref{LemEpiU} to $h=\tZ$. 
\end{proof}

\begin{lemma}\label{LemTMon}
For $\cV\subset\cU^0$, the topology $\cT_{\cV}$ is monoidal if and only if for every $A\in\ba$ and $u:U\to\unit$ in $\cV$, there exists $v:V\to\unit$ in $\bar{\cV}$ and a morphism $f:V\otimes A\to A\otimes U$ such that the following diagram is commutative
$$
\xymatrix{
A\otimes U\ar[rr]^{A\otimes u}&&A.\\
&V\otimes A\ar[lu]^{f}\ar[ru]_-{v\otimes A}
}$$
\end{lemma}
\begin{proof}
For the `only if' direction, we apply Theorem~\ref{ThmBiclosed}(c) for the case $X=\unit$. 
For fixed $A,u$, we can reformulate the existence of $v,f$ as in the lemma as demanding that $A\otimes R_u$ is in $\cT_{\cV}(A)$. Hence the existence of such $v,f$ are necessary for $\cT_{\cV}$ to be monoidal.

Conversely, assume that we always have $v,f$ as in the lemma. For arbitrary $R\in\cT_{\cV}(X)$, by definition there exists $u\in \cV$ with $u\otimes X\in R$. It follows that $v\otimes A\otimes X$ is in $A\otimes R$ and hence $A\otimes R\in \cT_{\cV}(A\otimes X)$. Since by the very construction of $\cT_{\cV}$ we always have $R\otimes A\in \cT_{\cV}(X\otimes A)$, this concludes the proof, again by Theorem~\ref{ThmBiclosed}(c).
\end{proof}

\begin{remark}
\begin{enumerate}[label=(\roman*)] 
\item If for every $u:U\to\unit$ in $\cV$, there exists an object $(U,\gamma: -\otimes U\stackrel{\sim}{\Rightarrow} U\otimes -)$ of the Drinfeld centre of $\ba$, then $\cT_{\cV}$ is monoidal.
\item For fixed $A\in\ba$ and $u: U\to\unit$ in $\cV$, the condition in Lemma~\ref{LemTMon} is equivalent with the condition that we can complete the following commutative square
$$\xymatrix{
\unit\ar[rr]^{\co_A}&&A\otimes A^\ast\\
V\ar@{-->}[u]\ar@{-->}[rr]&&A\otimes U\otimes A^\ast,\ar[u]_{A\otimes u \otimes A^\ast}
}$$
in a way that the left vertical arrow is in $\bar{\cV}$.\end{enumerate}
\end{remark}

\begin{prop}\label{PropUniMon}
Consider a subclass $\cV\subset\cU^0$ and assume that $\cT_{\cV}$ is monoidal, so that $\Sh(\ba,\cT_{\cV})$ is a biclosed Grothendieck category and $\tZ:\ba\to\Sh(\ba,\cT_{\cV})$ canonically monoidal. 
\begin{enumerate}[label=(\roman*)] 
\item For  another biclosed Grothendieck category $\bC$, composition with $\tZ$ yields an equivalence
$$[\Sh(\ba,\cT_{\cV}), \bC]^{\otimes}_{cc}\;\stackrel{\sim}{\to}\; [\ba,\bC]^\otimes_{\cV},$$
where $[\ba,\bC]^\otimes_{\cV}$ stands for the category of monoidal functors $F$ for which $F(\sigma_u)$ is exact in $\bC$, for every $u\in\cV$.
\item The sequence $\tZ(\sigma_u)$ is exact in $\Sh(\ba,\cT_{\cV})$ for every $u\in\cV$.
\item If $\mX$ is a scheme over $k$, then $[\Sh(\ba,\cT_{\cV}),\QCoh\mX]^{\otimes}_{cc}$ is equivalent to the category of monoidal functors $F:\ba\to\QCoh\mX$ for which $F(u)$ is an epimorphism, for every $u\in\cV$.
\end{enumerate}
\end{prop}
\begin{proof}
Part (i) is a special case of Theorem~\ref{ThmUniMon}. Part (ii) follows from Corollary~\ref{CorTriv}. Part (iii)  follows from part~(i) and Remark~\ref{RemOX}. 
\end{proof}





\begin{lemma}
For $\cV\subset\cU$, the following are equivalent:
\begin{enumerate}[label=(\alph*)]
\item $\cT_{\cV}=\cT_{\cU}$.
\item For every non-zero $u:U\to\unit$, there exists $v:V\to\unit$ in $\bar{\cV}$ and $f:V\to U$ with $v=u\circ f$.
\end{enumerate}
In that case we say that $\cV$ is {\bf dense}.
\end{lemma}
\begin{proof}
Since $\cT_{\cV}\subset\cT_{\cU}$, this is an almost immediate application of Corollary~\ref{Cor227}(i).
\end{proof}

\begin{remark}
In \cite[Analogy~3.2.3]{AbEnv}, a simplistic non-enriched analogue is given of $\Sh(\bD,\cT_{\cV})$ for $\cV\subset\cU^{ex}$.

\end{remark}

\subsection{Universal properties}

\begin{theorem}\label{Thm3case}
Let $\cT$ be a monoidal Grothendieck topology on $\ba$ for which $\Sh(\ba,\cT)\simeq\Ind\bT$ for a tensor category $\bT$.
Then we have $\cT=\cT_{\cV}$, with $\cV=\{u\in\cU\,|\, \tZ(u)\not=0\}$, and $J_{\cV}$ is the kernel of $\tZ$, so in particular $\ann_{\cV}(U)=\{u:U\to\unit\,|\, \tZ(u)=0\}$.
 Furthermore,  for any other tensor category $\bT_1$, composing with $\ba\to\bT$ induces an equivalence 
 $$[\bT,\bT_1]^{\otimes}_{ex}\;\,\stackrel{\sim}{\to}\,\;[\ba/J_{\cV},\bT_1]^{\otimes}_{faith}$$
 between the respective categories of monoidal functors which are exact or faithful.
\end{theorem}

\begin{proof}
By Theorem~\ref{ThmNew}(ii), that $\cT=\cT_{\cV}$ is an instance of Proposition~\ref{PropTQC}. That $J_{\cV}$ is the kernel of $\tZ$ follows from Lemma~\ref{PropsTV}(iii).


Composition with the relevant monoidal functors yield equivalences
$$[\bT,\bT_1]^{\otimes}_{ex}\;\simeq\;[\Ind\bT,\Ind\bT_1]^{\otimes}_{cc}\;\simeq\; [\ba,\Ind\bT_1]^\otimes_{\cV}.$$
Indeed, the right equivalence is Proposition~\ref{PropUniMon}(i). Furthermore, Theorem~\ref{ThmUniMon} establishes an equivalence between the category in the middle and the category of right exact monoidal functors $\bT\to\Ind\bT_1$. By \cite[Corollaire~2.10(i)]{Del90}, right exact monoidal functors $\bT\to \Ind\bT_1$ are automatically exact. Since the category of rigid objects in $\Ind\bT_1$ is equivalent to $\bT_1$, see~\cite[Lemma~1.3.7]{AbEnv}, the left equivalence follows.

The latter observation also shows that $[\ba,\Ind\bT_1]^\otimes_{\cV}$ is equivalent to the category of functors in $[\ba,\bT_1]^\otimes$ which send every sequence $\sigma_u, u\in\cV$, to an exact sequence.
By Theorem~\ref{ThmNew}(ii) the latter equals the subcategory of $F\in [\ba,\bT_1]^\otimes$ such that $F(v)\not=0$ (or equivalently $F(v)$ is surjective) for each $v\in\cV=\bar{\cV}$. By definition of $\ann_{\cV}$, the latter condition on $F$ also implies that $F(w)=0$ for each $w\in \ann_{\cV}$. By the above, we know that every morphism to $\unit$ in $\ba$ is either in $\cV$ or $\ann_{\cV}$. The condition on $F$ is thus that its kernel is precisely $J_{\cV}$. This concludes the proof.
\end{proof}

\begin{corollary}\label{CorEx}
The following are equivalent for $\cV\subset\cU^{ex}$ with $\cT_{\cV}$ monoidal.
\begin{enumerate}[label=(\roman*)] 
\item  $\ba$ is a full monoidal subcategory of a tensor category such that every object in the tensor category is a quotient of an object in $\ba$, and $\cV$ is dense.
\item $\Sh(\ba,\cT_{\cV})$ is the ind-completion of a tensor category.
\end{enumerate}
Furthermore, either condition implies $\cU=\cU^{ex}$.
\end{corollary}

\begin{remark}
By Theorem~\ref{Thm3case}, if for a topology $\cT$ we have $\Sh(\ba,\cT)\simeq\Ind\bT$, for a tensor category $\bT$, such that $\ba\to\bT$ is faithful and monoidal, then $\cT=\cT_{\cU}$. Furthermore, a necessary condition is $\cU=\cU^{ep}$ by Lemma~\ref{PropsTV}(iii). Finally, by Lemma~\ref{lemlim}, $\bT$ will be a tensor category over $k$ if $\ba(\sigma_u,\unit)$ is exact for every non-zero $u:U\to\unit$ in $\ba$, {\it i.e.} when
$$0\to\ba(\unit,\unit)\to\ba(U,\unit)\to\ba(U\otimes U,\unit)$$
is always exact.
\end{remark}

\begin{corollary}\label{CorLow}
Consider a monoidal functor $u:\ba\to\bT$ to a tensor category $\bT$ and denote its kernel by $J$. Assume that $u$ satisfies (F) and (G) in Lemma~\ref{LemLow}.
Then composition with $\ba/J\to\bT$ induced from $u$ yields, for every tensor category $\bT_1$, an equivalence
$$[\bT,\bT_1]^{\otimes}_{ex}\;\stackrel{\sim}{\to}\;[\ba/J,\bT_1]^{\otimes}_{faith}.$$
Assume that $u$ is faithful (or replace $\ba$ by $\ba/J$), then $\Ind\bT\simeq \Sh(\ba,\cT_{\cU})$ and for any biclosed Grothendieck category $\bC$, the functor
$$[\bT,\bC]_{ex}^{\otimes}\;\to\; [\ba,\bC]^{\otimes}$$
is fully faithful.
\end{corollary}
\begin{proof}
By Lemma~\ref{LemLow}, $\ba/J\to\bT$ is a creator. The result is thus an immediate application of Theorem~\ref{Thm3case} and Proposition~\ref{PropUniMon}(i).
\end{proof}

\begin{corollary}
Consider monoidal creators $u_i:\ba\to\bT_i$, for tensor categories $\bT_i$, with $i\in\{1,2\}$.
The kernels of $u_1$ and $u_2$ are either equal or incomparable for the inclusion order.
\end{corollary}
\begin{proof}
By adjunction, the kernels of $u_1,u_2$ are determined by the kernels of $\ba(-,\unit)\to\bT_i(-,\unit)$. Denote by $\cV_i$ the class of morphisms $U\to\unit$ which are {\em not} in the kernel of $\ba(-,\unit)\to\bT_i(-,\unit)$.  Assume the kernel of $u_1$ is contained in the kernel of $u_2$. This means that $\cV_1\supset \cV_2$, so also $\ann_{\cV_1} \supset \ann_{\cV_2}$. By Theorem~\ref{Thm3case} the kernels are given (up to the above adjuntion) by $\ann_{\cV_1}$ and $\ann_{\cV_2}$, so by assumption we also have $\ann_{\cV_1}\subset\ann_{\cV_2}$, whence the kernels of $u_1$ and $u_2$ are equal.
\end{proof}

%
%

\subsection{Graded modules and projective space}

\subsubsection{}\label{DefsGraded}
Let $S=\bigoplus_{i\in\mZ}S_i$ be a commutative and $\mZ$-graded $k$-algebra, such that $k\to S_0$ is an isomorphism, $S_1$ is finite dimensional and $S$ is generated in degree $1$. We exclude the trivial case $k=S$.

We consider the category $S\mbox{-gMod}$ of all $\mZ$-graded $S$-modules with homogeneous morphisms of degree zero, which is symmetric monoidal for $\otimes_S$. Then the full subcategory of finitely generated free modules constitutes a symmetric rigid monoidal category $\ba$ over $k$ with finite dimensional morphism spaces. The indecomposable objects in $\ba$ are the modules $S\langle j\rangle$, for $j\in\mZ$, which satisfy $S\langle j\rangle_i=S_{i-j}.$ We denote by $\cV\subset\cU(\ba)$ the morphisms to the unit~$S$ for which the cokernel in $S\mbox{-gMod}$ is finite dimensional.

\begin{prop}\label{PropQCoh1}
\begin{enumerate}[label=(\roman*)]
\item There are no fully faithful monoidal functors from $\ba$ to any tensor category.
\item Let $\mY$ be the projective $k$-scheme $\Proj S$, then $\QCoh\mY\simeq\Sh(\ba,\cT_{\cV})$.
\end{enumerate}
\end{prop}
\begin{proof}
 Clearly, every indecomposable object $M$ in $\ba$ is invertible (that is $M^\ast \otimes M \to \unit$ is an isomorphism). However there exist non-zero morphisms between non-isomorphic indecomposable objects. As invertible objects in tensor categories are necessarily simple, part (i) follows.
 
 It is known, see \cite[Chapter 3]{Se}, that $\QCoh\mY$ is the localisation of $S\mbox{-gMod}\simeq\PSh\ba$ with respect to the subcategory of torsion modules (unions of modules finite dimensional over $k$). By Theorem~\ref{ThmBQ} (and using Remark~\ref{RemEnv}), the topology $\cT$ corresponding to the localisation is given by all graded submodules $R\subset S\langle i\rangle$ for which $S\langle i\rangle/R$ is a torsion module, {\it i.e.} finite dimensional. Whence $\cT=\cT_{\cV}$, proving part (ii) by Remark~\ref{RemEnv}.
 \end{proof}
 
 \begin{prop}\label{PropP1}
 Consider  $S=k[x_0,x_1,\cdots, x_n]$ for some $n\in\mN$, with grading defined by $\deg x_i=1$. 
 \begin{enumerate}[label=(\roman*)]
 \item We have $\cV\subset\cU^{ex}\subset\cU^{ep}=\cU$, where the inclusion $\cU^{ex}\subset\cU^{ep}$ is strict.
 \item If $n=1$, then $\cV=\cU^{ex}$.
 \end{enumerate}
 \end{prop}
 

\begin{proof}
It is clear that every non-zero morphism $S\to ?$ in $\ba$ is a monomorphism (it is even a monomorphism in $S$-gMod).
It follows that $\cU^{ep}=\cU$.
Next, we claim that a non-zero morphism $u:U\to S$ is in $\cU^{ex}$ if and only if $\im u$ is not included in a non-trivial principal homogeneous ideal. Both parts then follow easily.

An arbitrary $u:U\to S$ can be written as follows. Consider a finite collection of natural numbers $\{n_i\}$ and define
$$u:U:=\bigoplus_{i}S\langle n_i\rangle\;\to\; S,\quad (p_i)_i\mapsto\sum_i a_ip_i,$$
for some $a_i\in S_{n_i}$. Then $\im u$ is not contained in any non-trivial principal ideal if and only if the greatest common denominator of $\{a_i\}$ is $1$.

Fix a morphism 
$$f:U\to S\langle j\rangle,\quad (p_i)_i\mapsto \sum_ib_ip_i$$ with $j\in\mZ$, defined by $b_i\in S_{n_i-j}$. Then $f\circ (U\otimes u-u\otimes U)=0$ if and only if $b_ia_l=b_la_i$ for all $i,l$. The latter is equivalent to the condition that there exists a rational function $q$ such that $q=b_i/a_i$ for all $i$.
On the other hand, we have that $f$ factors as $U\xrightarrow{u}S\to S\langle j\rangle$ if and only if $b_i/a_i$ is a polynomial, independent of $i$.

Now assume that the $a_i$ have greatest common denominator $a\not=1$ and set $b_i:=a_i/a$, so $b_i/a_i=1/a$. By the above paragraph, this yields a morphism $f:U\to S\langle \deg a\rangle$ which contradicts exactness of $\sigma_u$.
On the other hand, if the greatest common denominator of the $a_i$ is $1$, the above paragraph shows similarly that the sequence is exact. \end{proof}

\begin{lemma}\label{Lem1mor}
Let $S$ be as in Proposition~\ref{PropP1} and let $v:S\langle 1\rangle^{\oplus n+1}\to S$ be the morphism corresponding to $\{S\langle 1\rangle\to S:1\mapsto x_i\,|\,0\le i\le n\}$. Then $\cT_{\cV}=\cT_{\{v\}}$.
\end{lemma}
\begin{proof}
This application of Corollary~\ref{Cor227} is left as an exercise.
\end{proof}

The following is a categorical analogue of the usual description of morphisms $\mX\to\mP^n$.
\begin{prop} Assume $n>0$.
For any scheme $\mX$ we have an equivalence between 
$$[\QCoh\mP^n,\QCoh\mX]^{\otimes}_{cc}$$ and the following category. An object $\phi$ is a collection of morphisms $\{\phi_i:\cL\to\unit=\cO_{\mX}\,|\,0\le i\le n\}$ from an invertible sheaf $\cL\in \QCoh\mX$ for which the induced $\oplus_{i=0}^n\cL\to\unit$ is an epimorphism. A morphism from $\{\phi_i:\cL\to\unit\}$ to $\{\psi_i: \cL'\to \unit\}$ is an isomorphism $\cL\to \cL'$ leading to $n+1$ commutative diagrams with $\{\phi_i,\psi_i\,|\, 0\le i\le n\}$. 
\end{prop}
\begin{proof}
Let $S$ be as in Proposition~\ref{PropP1}, in particular $\mP^n=\Proj S$.
Note that $\ba$ is the universal (free) additive monoidal category over $k$ generated by an invertible object $S\langle 1\rangle$ and the $n+1$ morphisms $S\langle 1\rangle\to\unit$ from Lemma~\ref{Lem1mor}.  By Proposition~\ref{PropQCoh1}(ii) and Lemma~\ref{Lem1mor}, we have $\QCoh\mP^n\simeq \Sh(\ba,\cT_{\{v\}})$.
The conclusion now follows from Proposition~\ref{PropUniMon}(iii).
\end{proof} 

\appendix

\section{Generalising non-enriched topologies}\label{App}
In this appendix, $\bc$ always denotes a small (non-enriched) category. 

\subsection{Standard definitions}

We follow \cite[D\'efinition~II.1.3]{SGA4} and \cite[D\'efinition~II.1.1]{SGA4} for the definitions for non-enriched Grothendieck (pre)topologies.

\begin{definition}\label{DefNonEnr} A Grothendieck {\bf pretopology} on a category $\bc$ is an assignment to each $U\in\bc$ of a collection of sets of morphisms $\{U_i\to U\}$, called coverings of $U$, such that:
\begin{enumerate}
\item[(pt0)] If $\{U_i\to U\}$ is a covering and $f:V\to U$ any morphism, then all $U_i\times_UV$ exist.
\item[(pt1)] If $V\to U$ is an isomorphism, then $\{V\to U\}$ is a covering.
\item[(pt2)] If $\{U_i\to U\}$ is a covering and $f:V\to U$ any morphism in $\bc$, then $\{U_i\times_U V\to V\}$ is a covering.
\item[(pt3)] If $\{U_i\to U\}$ and $\{V_{ij}\to U_i\}$ are coverings, then $\{V_{ij}\to U_i\to V\}$ is a covering.
\end{enumerate}
\end{definition}

\begin{definition}
A Grothendieck {\bf topology} on a category $\bc$ is an assignment to each $U\in\bc$ of a collection $\cT(U)$ of subfunctors of $\bc(-,U)$ (sieves on $U$) such that for every $U\in\bc$:
\begin{enumerate}
\item[(T1)] We have $\bc(-,U)\in \cT(U)$;
\item[(T2)] For $R\in \cT(U)$ and a morphism $f:V\to U$  in $\bc$, we have $f^{-1}R\in \cT(V)$;
\item[(T3)] For a sieve $S$ on $U$ and $R\in\cT(U)$ such that for every $V\in\bc$ and $f\in R(V)\subset \bc(V,U)$ we have $f^{-1}S \in \cT(V)$, it follows that $S\in\cT(U)$.
\end{enumerate}
\end{definition}

\subsubsection{}\label{tpt} The topology associated to a pretopology assigns to every $U\in\bc$ the collection of sieves on $U$ which contain a sieve generated by a cover $\{U_i\to U\}$ in the pretopology, see \cite[Proposition~II.1.4]{SGA4}.

 A presheaf $F:\bc^{\op}\to \Set$ is a sheaf for a topology, if 
$F(U)\to \Nat(R,F)$ is an isomorphism for every $U\in\bc$ and every sieve $R\in\cT(U)$. For a given pretopology, a presheaf $F$ is a sheaf if 
\begin{equation}\label{classsheaf2}F(U)\;\to \;\prod_i F(U_i)\;\rightrightarrows\;\prod_{j,l}F(U_j\times_UU_l)\end{equation}
is an equaliser for every covering $\{U_i\to U\}$. These definitions of sheaves agree when we consider a pretopology and its associated topology, see \cite[Corollaire~II.2.4]{SGA4}.

\subsection{Generalisations}

Now we explain how one can gradually loosen Definition~\ref{DefNonEnr} in several directions, while keeping the story in \ref{tpt} intact.

\subsubsection{Omitting (pt1) and (pt3)}\label{rempt13}
Clearly, one can remove (pt1) from the definition of a pretopology and nothing in~\ref{tpt} changes. 

Furthermore, we claim it is possible to generalise further by omitting (pt3) and just compensate by now defining the associated topology to be the one consisting of all sieves which contain compositions of coverings (by compositions of coverings we mean iterations of the type of composition in (pt3)).

Indeed, clearly property (pt0) and (pt2) imposed on coverings extend, by iteration, to arbitrary compositions of coverings. In other words, we can easily construct a genuine pretopology out of  a class of coverings satisfying only (pt0) and (pt2). Furthermore, elementary diagram chasing reveals that sheaves for the associated topology are still the presheaves which satisfy equation~\eqref{classsheaf2}, where we only need to consider coverings (not compositions of coverings).

\subsubsection{Weak pullbacks}
Recall that for a diagram $B\to A\leftarrow C$, a weak pullback is an object~$D$ with morphisms $B\leftarrow D\to C$, so that 
$$\bc(E,D)\;\to\; \bc(E,B)\times_{\bc(E,A)}\bc(E,C)$$
is surjective for every $E\in\bc$.

One can consider another relaxation of Definition~\ref{DefNonEnr}, we keep (pt1) and (pt3) but change (pt0) and (pt2) to
\begin{enumerate}
\item[(pt0')] If $\{U_i\to U\}$ is a covering and $f:V\to U$ any morphism, then every $U_i\to U\leftarrow V$ admits at least one weak pullback.
\item[(pt2')] If $\{U_i\to U\}$ is a covering and $f:V\to U$ any morphism in $\bc$, then $\{U_i'\to V\}$ is a covering for some weak pullbacks $U_i'$ of $U_i\to U\leftarrow V$.
\end{enumerate}

One can easily check that the associated system of sieves (defined just as in \ref{tpt}) yields again a topology (proving this does not even require the fact that $U_i'$ in (pt2') is a weak pullback). Furthermore, one can show, following the classical approach from \cite{SGA4}, that the sheaves for the topology can now be characterised as those presheaves~$F$ for which 
\begin{equation}\label{EqWeak}F(U)\;\to \;\prod_i F(U_i)\;\rightrightarrows\;\prod_{j,l}F(U_{jl})\end{equation}
is an equaliser for every covering and an arbitrary(!) choice $U_{jl}$ of pullback for $U_i\to U\leftarrow U_j$.

\subsubsection{Weak pullback property `up to cover'}

In a further generalisation, we can omit even existence of weak pullbacks and instead augment the definition of a covering to include objects which are not necessarily weak pullbacks, but act as such `up to coverings'.

\begin{definition}
An {\bf extended covering } of $U\in\bc$ is a set $I$ with for each $(i,j)\in I\times I$ a commutative diagram
$$\xymatrix{
U_{ij}\ar[r]\ar[rd]& U_i\ar[rd]\\
&U_j\ar[r]&U.}$$
\end{definition}

We combine this idea together with the generalisation in \ref{rempt13}:

\begin{definition}\label{DefGPT}
A {\bf generalised pretopology} on $\bc$ is a class of extended coverings such that
\begin{enumerate}
\item[(pta)] For an extended covering $\{U_i\leftarrow U_{ij}\to U_j, U_i\to U\,|\, i,j\in I\}$ in the class and a morphism $f:V\to U$ there exists an extended covering $\{V_i\leftarrow V_{ij}\to V_j, V_i\to V\,|\, i,j\in I\}$ yielding commutative squares
$$\xymatrix{
U_i\ar[rr]&& U\\
V_i\ar[rr]\ar@{-->}[u]&&V.\ar[u]}$$
 \item[(ptb)] For an extended covering $\{U_i\leftarrow U_{ij}\to U_j, U_i\to U\,|\, i,j\in I\}$ in the class and a commutative square for some given $i_0,j_0\in I$
 $$\xymatrix{U_{i_0}\ar[rr]&& U\\
 V\ar[u]\ar[rr]&&U_{j_0},\ar[u]}$$
 there exists an extended covering of $V$ such that, for each $V_l\to V$ in the covering, the composite morphisms $V_l\to U_{i_0}$ and $V_l\to U_{j_0}$ both factor via the same morphism $V_l\to U_{i_0j_0}$.
\end{enumerate}
\end{definition}

\begin{theorem}
Consider a generalised pretopology on $\bc$.
\begin{enumerate}
\item The collection of sieves which contain some sieve generated by compositions of the coverings in the extended coverings,  forms a topology on $\bc$. 
\item A presheaf $F:\bc^{\op}\to \Set$ is a sheaf for the topology in (1) if and only if
$$F(U)\to \prod_i F(U_i)\rightrightarrows \prod_{jl} F(U_{jl})$$
is an equaliser for each extended covering in the generalised pretopology.
\end{enumerate}
\end{theorem}
\begin{proof}
Part (1) is straightforward (and does not use condition (ptb)). Part (2) follows from a technical argument, which is essentially a simplified version of the proof in Theorem~\ref{SecSheaves}.
\end{proof}

\subsection{Issues regarding linearisation}\label{Issues}

\subsubsection{}\label{Pres} For a given covering $\{U_i\to U\}$, the sieve it generates is given by the co-equaliser of
$$\bigsqcup_{jl}\left(\bc(-,U_j)\times_{\bc(-,U)}\bc(-,U_l)\right)\;\rightrightarrows\; \bigsqcup_i\bc(-, U_i),$$
which is a direct consequence of the fact that coproducts and pullbacks commute in $\Set$.

This observation is the key to the proof that condition~\eqref{classsheaf2}, or more generally~\eqref{EqWeak}, describes sheaves for the associated topology.

\subsubsection{} In $\Ab$, coproducts and pullbacks do not commute. As a consequence, for morphisms of abelian groups $\{B_i\to B\,|\,i\in I\}$, the sequence
$$\bigoplus_{kl}\left(B_k\times_BB_l\right)\;\to\; \bigoplus_i B_i\;\to\; B,$$
with the first morphism the difference of the two projections, is not necessarily exact if $|I|>2$, see Example~\ref{ExBB} below. This implies that in an additive category, even the assumption of the existence of pullbacks is not enough to secure a canonical presentation for the sieve generated by a covering. Furthermore, linearisation does not preserve pullbacks, see~\ref{kFibre}.

In case of finite coverings, passing to the additive envelope, one can reduce finite coverings to singleton coverings, but for infinite coverings even that is not possible. This is what motivates us to take an approach to additive pretopologies which linearises Definition~\ref{DefGPT} rather than Definition~\ref{DefNonEnr}.

\begin{example}\label{ExBB}
Consider the morphism of abelian groups 
$$\mZ^3\to \mZ^2,\quad (a,b,c)\mapsto (a-b,a-c).$$
If we let $B_1,B_2,B_3$ be the three copies of $\mZ$ in $\mZ^3$ and $B=\mZ^2$, then $B_i\times_B B_j=0$ for all $i\not=j$.
\end{example}

\subsection{Linearising unenriched pretopologies}\label{SecLinearise}
For a commutative ring $k$, we can consider the $k$-linear category $k\bc$. We have $\Ob (k\bc)=\Ob\bc$ and $(k\bc)(X,Y)=k(\bc(X,Y))$.
\subsubsection{}\label{kFibre}

For a pullback $A\times_C B$ in $\bc$, the same object in $k\bc$ need not be a pullback. However, for a morphism $A\to B$ in $\bc$, the pullback $A\times_BA$ is a weak pullback in $k\bc$. This is a direct consequence of the observation that for a function of sets $S\to T$ we have a canonical surjection
\begin{equation}\label{eqST}k(S\times_TS)\;\tto\; kS\times_{kT}kS,\qquad e_{(s_1,s_2)}\mapsto (e_{s_1}, e_{s_2}) .\end{equation}


\subsubsection{}\label{Enrich} Consider a (unenriched) Grothendieck pretopology on $\bc$. For each covering $\{U_i\to U\}$ in the pretopology, we consider the formal sequence
$$\amalg_{j,l}(U_j \times_U U_l)\;\to\;\amalg_i U_i \;\to\; U$$
in $k\bc$ (with pullback to be interpreted in $\bc$). The morphism $c=(c_i):U_j \times_U U_l\to\amalg_i U_i$ is such that $c_j$ is the projection $ U_j \times_U U_l\to U_j$, while $c_l$ is $-1$ times the projection  $ U_j \times_U U_l\to U_l$ and $c_i=0$ when $i\not\in\{j,l\}$.

Denote the class of all sequences constructed this way by $\cS$.

\begin{theorem}
The class $\cS$ from \ref{Enrich} is a $k$-linear Grothendieck pretopology on $k\bc$.
\end{theorem}
\begin{proof}
We start by proving (PTa). Consider a covering $\{U_i\to U\}$ and a morphism $f\in k\bc(A,U)$. If $f=\lambda f_0$ for $f_0\in \bc(A,U)$ and $\lambda\in k$, then the condition in (PTa) follows easily from (pt2) with $\amalg_\delta C_\delta:=\amalg_i (U_i\times_UA)$ where the pullbacks are taken with respect to $f_0$. Condition (PTa) for arbitrary morphisms $f=\sum_i \lambda_if_i$ with $\lambda_i\in k$ and $f_i\in\bc(A,U)$ then follows from an iterative argument similar to the one in the proof of Lemma~\ref{LemPTa}(i).
 

That (PTb'), so in particular (PTb), holds follows from the observations in \ref{Pres} and the surjection~\eqref{eqST}.
 \end{proof}
%


\begin{remark}\label{RemLinTop}
One can consider the canonical linearisation of the unenriched topology corresponding to the pretopology, which is the $k$-linear topology on $k\bc$ comprising all sieves which contain a trivial linearisation of a sieve in the unenriched topology. 
The latter equals the $k$-linear topology $\topp\cS$, with $\cS$ as in~\ref{Enrich}. 
\end{remark}
\begin{example}
Let $\bc$ be the category of open subspaces of a topological space $X$. It is symmetric monoidal for $U\otimes V=U\times V=U\cap V$. We can consider the usual pretopology of open coverings. The corresponding linearised topology on $k\bc$ is monoidal, for instance by application of Remark~\ref{RemLinTop} and Theorem~\ref{ThmSym}(b). The corresponding localisation of 
$$\PSh(k\bc)\;\simeq\;\Fun(\bc^{\op},\Mod)\;\simeq\;\PSh(X;k)$$ is of course the symmetric monoidal category of sheaves of $k$-modules on $X$.
\end{example}

\subsection*{Acknowledgement}
The author thanks Pavel Etingof, Victor Ostrik and Bregje Pauwels for many interesting discussions. Several ideas in Section~\ref{SecPT} owe their existence to \cite{Sch1, Sch2}. The research was supported by ARC grant DP180102563.


\begin{thebibliography}
	{EGNO}


\bibitem[AR]{AR}
J.~Ad\'amek, J.~Rosick\'y:
Reflections in locally presentable categories.
Arch. Math. (Brno) 25 (1989), no. 1-2, 89--94. 

\bibitem[AGV]{SGA4} M. Artin, A. Grothendieck, and J.-L. Verdier, Th\'eorie des topos et cohomologie \'etale des sch\'emas, 1, Springer-Verlag, Berlin, 1972--73, 
(SGA4). Avec la collaboration de N. Bourbaki, P. Deligne et B. Saint-Donat, Lecture Notes in Mathematics, Vol. 270. 

\bibitem[BQ]{BQ} F.~Borceux, C.~Quinteiro: A theory of enriched sheaves. Cahiers Topologie G\'eom. Diff\'erentielle Cat\'eg. 37 (1996), no. 2, 145--162.



\bibitem[Co1]{Ideals} K. Coulembier: Tensor ideals, Deligne categories and invariant theory. Selecta Math. (N.S.) 24 (2018), no. 5, 4659--4710.

\bibitem[Co2]{AbEnv} K. Coulembier: Monoidal abelian envelopes. To appear in Compositio Mathematica. arXiv:2003.10105.




\bibitem[Da]{Day} B.~Day: A reflection theorem for closed categories. J. Pure Appl. Algebra 2 (1972), no. 1, 1--11.
	
	
	\bibitem[DM]{DM} P.~Deligne, J.S.~Milne: Tannakian Categories. In
Hodge cycles, motives, and Shimura varieties. 
Lecture Notes in Mathematics, 900. Springer-Verlag, Berlin-New York, 1982, pp. 101-228.
	
	\bibitem[De]{Del90} P.~Deligne: Cat\'egories tannakiennes. The Grothendieck Festschrift, Vol. II, 111--195, Progr. Math., 87, Birkh\"auser Boston, Boston, MA, 1990. 



\bibitem[EGNO]{EGNO}P.~Etingof, S.~Gelaki, D.~Nikshych, V.~Ostrik:
Tensor categories. 
Mathematical Surveys and Monographs, 205. American Mathematical Society, Providence, RI, 2015. 



\bibitem[EHS]{EHS}
I.~Entova-Aizenbud, V.~Hinich, V.~Serganova:
Deligne categories and the limit of categories~$Rep(GL(m|n))$.
Int. Math. Res. Not. IMRN 2020, no. 15, 4602--4666.



\bibitem[Fr]{Freyd}  P.~Freyd: Abelian categories. An introduction to the theory of functors. Harper's Series in Modern Mathematics Harper \& Row, Publishers, New York 1964.

\bibitem[IK]{IK}G.B.~Im, G.M.~Kelly: A universal property of the convolution monoidal structure. J. Pure Appl. Algebra 43 (1986), no. 1, 75--88.

\bibitem[Ke]{Ke} G.M.~Kelly: Basic Concepts of Enriched Category Theory.
Cambridge University Press, Lecture Notes in Mathematics 64, 1982.

\bibitem[Lo]{Lo}
W.~Lowen: A generalization of the Gabriel-Popescu theorem. J. Pure Appl. Algebra 190 (2004), no. 1-3, 197--211.

\bibitem[LRS]{LRS}
W.~Lowen, J.~Ramos Gonz\'alez, B.~Shoikhet: On the tensor product of linear sites and Grothendieck categories. Int. Math. Res. Not. IMRN 2018, no. 21, 6698--6736.



\bibitem[Ma]{Mac}S.~MacLane:
Categories for the working mathematician. 
Graduate Texts in Mathematics, Vol. 5. Springer-Verlag, New York-Berlin, 1971.

		\bibitem[PG]{PG}
N.~Popesco, P.~Gabriel:
Caract\'erisation des cat\'egories ab\'eliennes avec g\'en\'erateurs et limites inductives exactes.
C. R. Acad. Sci. Paris 258 (1964), 4188--4190. 

\bibitem[Sc1]{Sch1} D.~Sch\"appi: Ind-abelian categories and quasi-coherent sheaves. Math. Proc. Cambridge Philos. Soc. 157 (2014), no. 3, 391--423. 


\bibitem[Sc2]{Sch2}D.~Sch\"appi: Constructing colimits by gluing vector bundles.
 Adv. Math. 375 (2020), 107394, 85 pp.


\bibitem[Se]{Se}J.P.~Serre: Faisceaux algébriques cohérents. Ann. of Math. (ii) 61 (1955), 197--278.



\bibitem[Ta]{Ta}
G.~Tamme: Introduction to \'etale cohomology. Universitext. Springer-Verlag, Berlin, 1994.

 
 	\end{thebibliography}
\end{document}